\documentclass[a4paper,12pt]{article}
\usepackage[english]{babel}
\usepackage[utf8x]{inputenc}
\usepackage{amsmath,amsfonts,amssymb,amsthm}
\usepackage{mathtools}
\usepackage{hyperref}
\usepackage{enumitem}
\usepackage{float}
\usepackage{xcolor}
\usepackage{tikz-cd}
\usepackage{titlesec}
\usepackage{graphicx}
\usepackage[all,cmtip]{xy}
\usepackage[english]{babel}
\usepackage{textcomp}
\usepackage{blindtext}
\usepackage{pifont}
\usepackage{subcaption}
\usepackage{float}

\def    \C      {{\mathbb C}}
\def    \R      {{\mathbb R}}
\def \Z {{\mathbb Z}}

\renewcommand{\epsilon}{\varepsilon}

\DeclareMathOperator{\interior}{int}

\newtheorem{theorem}{Theorem}[section]

\newtheorem{lemma}[theorem]{Lemma}
\newtheorem{prop}[theorem]{Proposition}
\newtheorem{rmk}[theorem]{Remark}
\newtheorem{ex}[theorem]{Example}

\AtEndDocument{\bigskip{\footnotesize
  \textsc{Vinicius G. B. Ramos, Instituto de Matem\'atica Pura e Aplicada, Rio de Janeiro, Brazil.} \par  
  \textit{E-mail}:  \texttt{ vgbramos@impa.br}\\
  \textsc{Brayan Ferreira, Universidade Federal do Espírito Santo, Vitória, Brazil.} \par  
  \textit{E-mail}:  \texttt{brayan.ferreira@ufes.br}\\
   \textsc{Alejandro Vicente, Hebrew University of Jerusalem, Jerusalem, Israel.} \par  
  \textit{E-mail}:  \texttt{ kvicente931207@gmail.com}
  }}

\providecommand{\keywords}[1]
{
  \small	
  \textbf{\textbf{Keywords: }} #1
}

\title{Gromov width of the disk cotangent bundle of spheres of revolution}
\date{}
\author{Brayan Ferreira, Vinicius G. B. Ramos and Alejandro Vicente}
\begin{document}

\maketitle

\begin{abstract}
Inspired by work of the first and second authors, this paper studies the Gromov width of the disk cotangent bundle of spheroids and Zoll spheres of revolution. This is achieved with the use of techniques from integrable systems and embedded contact homology capacities.
\end{abstract}


\section{Introduction}

Symplectic embedding problems are a central subject in symplectic geometry since Gromov's nonsqueezing theorem \cite{gromov1985pseudo}. See \cite{Schlenk2017SymplecticEP} for a thorough account on results about symplectic embedding problems. Among embedding problems, the most basic question is about the largest ball that can be symplectically embedded into a given manifold. More precisely, given a $2n$-dimensional symplectic manifold $(X,\omega)$, its Gromov width $c_{Gr}(X,\omega)$ is defined to be the supremum of $a>0$ such that there exists a symplectic embedding $(B^{2n}(a),\omega_0)\hookrightarrow (X,\omega)$, where \[B^{2n}(a)=\left\{(x_1,y_1,\dots,x_n,y_n)\in\R^{2n}\mid \sum_{i=1}^n (x_i^2+y_i^2)\leq a/\pi\right\}\text{ and }\omega_0=\sum_{i=1}^n dx_i\wedge dy_i.\]

For any smooth manifold $Q$, its cotangent bundle $T^*Q$ has a canonical symplectic structure $\omega_{can}$ which can be written in local coordinates as
\[\omega_{can}=\sum_i dp_i\wedge dq_i.\]
A Riemannian structure on $Q$ gives rise to a norm on the fibers of $T^*Q$ and so we can define the disk cotangent bundle $D^* Q$ as
\[D^* Q=\left\{(q,p)\in T^* Q\mid \Vert p\Vert< 1\right\}.\]

Let $S^2\subset \R^3$ denote the unit sphere endowed with the Riemannian metric induced from $\R^3$. In \cite{ferreira2021symplectic}, the first and second authors showed that \[c_{Gr}(D^*S^2,\omega_{can})=2\pi.\]
The goal of this article is to generalize this result to spheres of revolution, which are either an ellipsoid or Zoll. We will now explain both cases.

For $a,b,c>0$, let $\mathcal{E}(a,b,c)\subset \R^3$ be the ellipsoid defined by the equation:
$$\frac{x^2}{a^2}+\frac{y^2}{b^2}+\frac{z^2}{c^2}=1.$$
When the two parameters $a,b$ coincide, we get an ellipsoid of revolution, also called a spheroid. Up to a normalization, we can assume that $a=b=1$. We say that $\mathcal{E}(1,1,c)$ is oblate when $c<1$ and prolate if $c>1$. We also note that $S^2= \mathcal{E}(1,1,1)$. In order to state the Gromov width of $D^* E(1,1,c)$, we need to use elliptical integrals. We now recall the definition\footnote{Note that this differs slightly from the standard definitions of the complete elliptic integrals, more specifically, we take $k$ instead of $k^2$. Moreover, these functions take real values for all $k\in(-\infty,1)$.} of the complete elliptical integrals of the first, second and third kinds. For $k,n<1$, let
\begin{equation}\label{defellipint}
\begin{aligned}
K(k)&=F\left(\frac{\pi}{2}\Bigm\vert k\right)=\int_0^{\pi/2}\frac{d\theta}{\sqrt{1-k\sin^2\theta}},\\
E(k)&=E\left(\frac{\pi}{2}\Bigm\vert k\right)=\int_0^{\pi/2}\sqrt{1-k\sin^2\theta}\,d\theta,\\
\Pi(n,k)&=\Pi\left(n;\frac{\pi}{2}\Bigm\vert k\right)=\int_0^{\pi/2}\frac{d\theta}{(1-n\sin^2\theta)\sqrt{1-k\sin^2\theta}}.
\end{aligned}
\end{equation}
For each $c<1/2$, let $j_0(c)$ be the smallest $j\in[0,1]$ such that 
\begin{equation}\label{eq:j0c}-j c K\left(\left(1-\frac{1}{c^2}\right)(1-j^2)\right)+\frac{j}{c}\Pi\left(1-\frac{1}{c^2},\left(1-\frac{1}{c^2}\right)(1-j^2)\right) =\frac{\pi}{4}.\end{equation}

The left hand side of \eqref{eq:j0c} equals $0$ when $j=0$ and $(1-c)\pi/2$, when $j=1$. It follows from the Intermediate Value Theorem that \eqref{eq:j0c} has a solution for a fixed $c\le 1/2$. We will see in the proof of Lemma \ref{mergulho} that this solution is unique. Now let 
\[
\begin{aligned}
\alpha(c)&=8cE\left(\left(1-\frac{1}{c^2}\right)(1-j_0(c)^2)\right),\text{ for }c<1/2,\\
\beta(c)&= 4E(1-c^2), \text{ for }c>1.
\end{aligned}\]

\begin{rmk}\label{rmk:length}
The number $\beta(c)$ is simply the length of any meridian of $\mathcal{E}(1,1,c)$. We will see in Section \ref{sec:echcap} that $\alpha(c)$ is the length of a simple geodesic on $\mathcal{E}(1,1,c)$ that intersects the equator 4 times and only exists when $c<1/2$.
\end{rmk}
We also note that the function $\beta(c)$ is injective (in fact, increasing in $c$) and that $\beta(1) = 2\pi$. Moreover, it is simple to check that $\pi(1+c) \leq \beta(c) \leq 2\pi \sqrt{1+c^2}$. In particular, we have the well defined number $\beta^{-1}(4\pi) \in (2,3]$. We can now state the main result of this paper.

\begin{theorem}\label{thm:main}
The Gromov width of $D^*\mathcal{E}(1,1,c)$ is given by
\begin{equation*}c_{Gr}(D^*\mathcal{E}(1,1,c),\omega_{can})=\left\{\begin{aligned} \alpha(c),&\text{ for }0<c<1/2,\\
2\pi,&\text{ for }1/2\le c \le 1,\\
\beta(c),&\text{ for }1<c<\beta^{-1}(4\pi),\\
4\pi,&\text{ for }c\ge \beta^{-1}(4\pi).
\end{aligned}\right.\end{equation*}
\end{theorem}

\begin{figure}[H]
\centering
\includegraphics[scale=0.6]{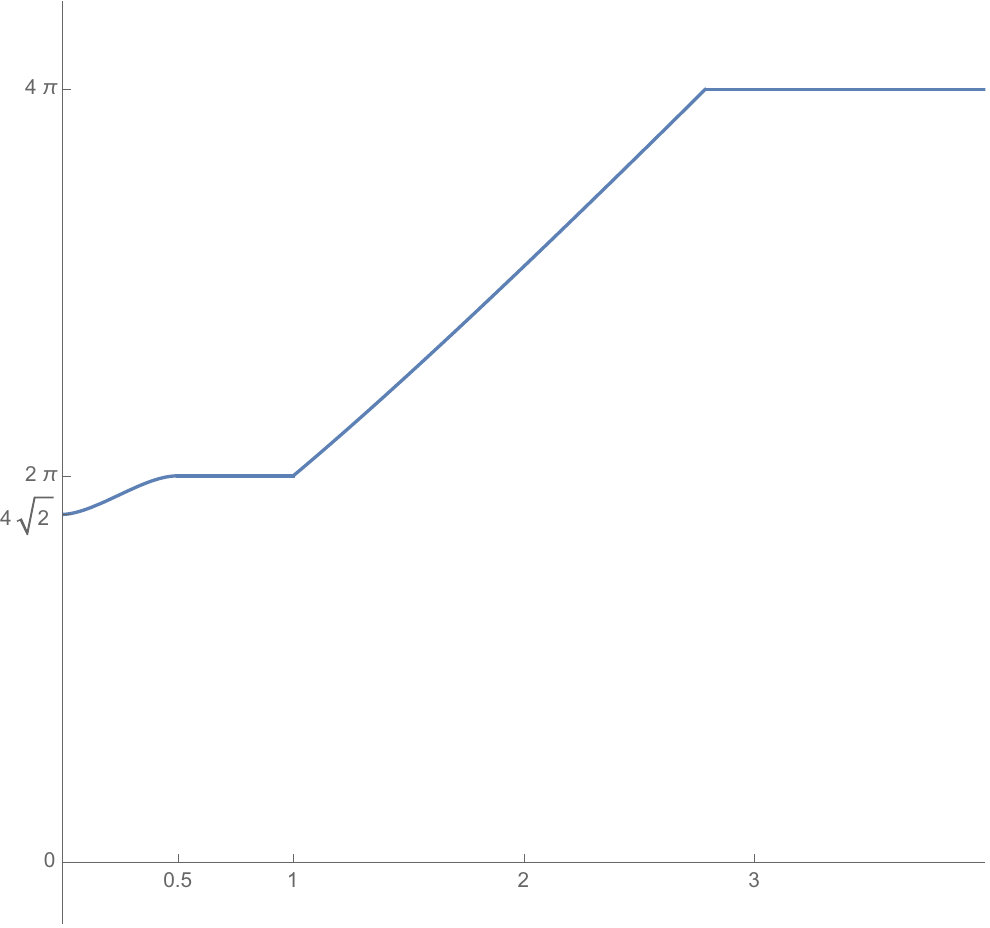}
\caption{The graph of function $c\mapsto c_{Gr}(D^*\mathcal{E}(1,1,c),\omega_{can})$.}
\end{figure}

The proof of Theorem \ref{thm:main} has two parts. We will first find a toric domain that symplectically embeds into $D^*\mathcal{E}(1,1,c)$ filling its volume and we will find a ball embedding into this toric domain. Then we will use embedded contact homology (ECH) capacities to show that we cannot do better than that.

Now recall that a Riemannian manifold is said to be \textit{Zoll} if all geodesics are closed and have the same length. Our second result is the following.

\begin{theorem}\label{thm:grzoll}
Let $S$ be a Zoll sphere of revolution and $\ell$ be the length of any simple closed geodesic. Then \[c_{Gr}(D^*S,\omega_{can}) = \ell.\]
\end{theorem}
The proof of Theorem \ref{thm:grzoll} is similar to Theorem \ref{thm:main}. Using the integrability of the geodesic flow, we can find an embbeded ball in the disk cotangent bundle of more general spheres of revolution, obtaining a lower bound for the Gromov width for a larger class of spheres of revolution. The upper bound in the Zoll case is again obtained by ECH capacities.

\begin{rmk}
In fact, all the unit sphere cotangent bundles of Zoll metrics $S^*S = \partial D^*S$ are strictly contactomorphic and there exists a contactomorphism that is isotopic to the inclusion into $T^*S^2\setminus S^2$, see \cite[Theorem B.1]{abbondandolo2017systolic}. The $1$-homogeneous extension of this contactomorphism is an exact symplectomorphism of $T^*S^2\setminus S^2$ that is generated by a $1$-homogeneous Hamiltonian. By cutting off the Hamiltonian close to the zero section $S^2 \subset T^*S^2$, one gets a symplectomorphism of $T^*S^2$ into itself which coincides with the previous one away from the zero section. In particular, all the disk cotangent bundles of Zoll metrics $D^*S$ are symplectomorphic, and hence, all the symplectic capacities coincide, not only the Gromov width. Nevertheless, the proof of Theorem \ref{thm:grzoll} uses the computation of ECH capacities for the Zoll case which can be worked out without appealing to this strict contactomorphism.
\end{rmk}

In the last quinquennium, computations of symplectic capacities in disk cotangent bundles of surfaces have been carried out by several authors. More specifically, in \cite{Bimmermann2023HoferZehnderCO}, Bimmermann determines the Hofer-Zehnder capacity of magnetic disc tangent bundles over constant curvature surfaces. In \cite{brocic2025riemannian}, Broćić computes a version of Gromov width relative to Lagrangians for disk cotangent bundles of closed Riemannian manifolds.

\subsection{A toric domain in disguise}\label{sec: toricdisg}
A four-dimensional toric domain is a subset of $\C^2$ defined by
\[\mathbb{X}_\Omega=\left\{(z_1,z_2)\in \C^2\mid(\pi|z_1|^2,\pi|z_2|^2)\in\Omega\right\},\] where $\Omega\subset\R_{\ge 0}^2$ is the closure of an open set in $\R^2$. Now suppose that $\Omega$ is the region bounded by the coordinate axes and a piecewise smooth curve $\gamma$ connecting $(a,0)$ to $(0,b)$ for some $a,b>0$. A toric domain $\mathbb{X}_\Omega$ is said to be
\begin{enumerate}[label=(\alph*)]
\item concave, if $\gamma$ is the graph of a convex function $[0,a]\to\R_{\ge 0}$;
\item convex, if $\gamma$ is the graph of a concave function $[0,a]\to\R_{\ge 0}$;
\item weakly convex, if $\Omega$ is a convex set.
\end{enumerate}
See the third domain in Figure \ref{fig: domains_c} for an example of a toric domain which is weakly convex but not convex.

The first step in the proof of Theorem \ref{thm:main} is to find a toric domain that symplectically embeds into $D^*\mathcal{E}(1,1,c)$. More precisely, we use the main idea of \cite{Ramos2017SymplecticEA} as in \cite{ferreira2021symplectic}, to prove that the complement of the fiber over $(0,0,c)$ is symplectomorphic to the interior of a toric domain. It turns out that this actually holds in a more general setting.

Let $S \subset \R^3$ be a sphere of revolution, meaning that $S$ is a genus zero compact smooth surface, which is invariant under rotation around a fixed coordinate axis. Without loss of generality, we can assume that $S$ has the form
\begin{equation}\label{defS}
S = \{(u(z)\cos \theta, u(z) \sin \theta, z)\mid  z \in [a,b], \theta \in \R/2\pi \Z\},
\end{equation}
where $u(z)>0$ for $z\neq a,b$, and $u(a) = u(b) = 0$. We define the north and south poles by $P_N = (0,0,b)$ and $P_S=(0,0,a)$, respectively. An equator is a circle $C=\{(x,y,z)\in S \mid z=z_0\}$, whenever $z_0$ is a critical point of $u$.

\begin{theorem}\label{toricspheres}
Let $S\subset \mathbb{R}^3$ be a sphere of revolution  with a unique equator. Then there exists a toric domain $\mathbb{X}_{\Omega_S}$ such that $(D^*(S \backslash \{P_N\}),\omega_{can})$ is symplectomorphic to $(\interior \mathbb{X}_{{\Omega}_S},\omega_0)$.
\end{theorem}

For the case of Zoll spheres of revolution we can go further and give the precise toric domain found in Theorem \ref{toricspheres}.

\begin{prop}\label{zollbidisk}
Let $S\subset \mathbb{R}^3$ be a Zoll sphere of revolution and $\ell$ be the length of any simple closed geodesic. Then $(D^*(S \backslash \{P_N\}),\omega_{can})$ is symplectomorphic to the symplectic bidisk $(\interior B^2(\ell) \times B^2(\ell),\omega_0)$.
\end{prop}

Note that combining Stokes' theorem and Fubini's theorem, one can conclude that $\text{area}(S) = \frac{\text{vol}(D^*S,\omega_{can})}{\pi}$, and hence, it follows from the previous result that
$$\text{area}(S) = \frac{\text{vol}(B^2(\ell) \times B^2(\ell),\omega_0)}{\pi} = \frac{\ell^2}{\pi}$$
holds whenever $S \subset \R^3$ is a Zoll sphere of revolution. This recovers a well known result about the volume of Zoll manifolds due to Weinstein in \cite{weinstein1974volume} for the family of Zoll spheres of revolution.

We can also give the precise toric domain corresponding to the ellipsoids of revolution, as well as analyze its convexity/concavity.

\begin{prop}\label{thm:toric}
For each $c>0$, $(D^*(\mathcal{E}(1,1,c)\backslash \{(0,0,c)\}),\omega_{can})$ is symplectomorphic to $(\interior \mathbb{X}_{\Omega_c},\omega_0)$, where $\mathbb{X}_{\Omega_c}$ is a toric domain which is:
\begin{enumerate}[label=(\roman*)]
\item neither concave nor weakly convex for $c<1$,
\item weakly convex for $c\geq 1$.
\end{enumerate}
\end{prop}

\begin{figure}[H]
\centering
\begin{subfigure}[t]{0.32\textwidth}
\centering
\includegraphics[scale=0.35]{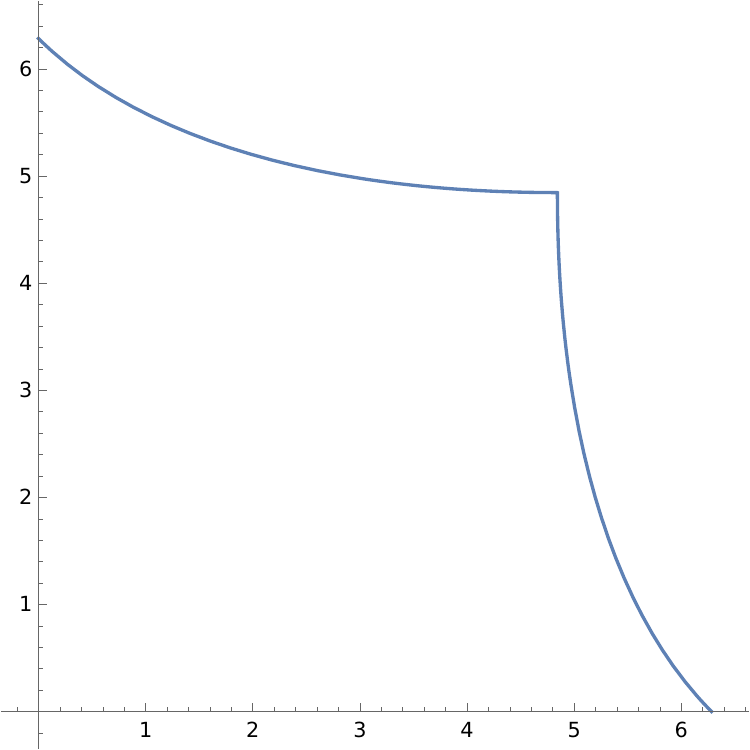}
\caption{$c=0.5$}
\end{subfigure}
\begin{subfigure}[t]{0.32\textwidth}
\centering
\includegraphics[scale=0.35]{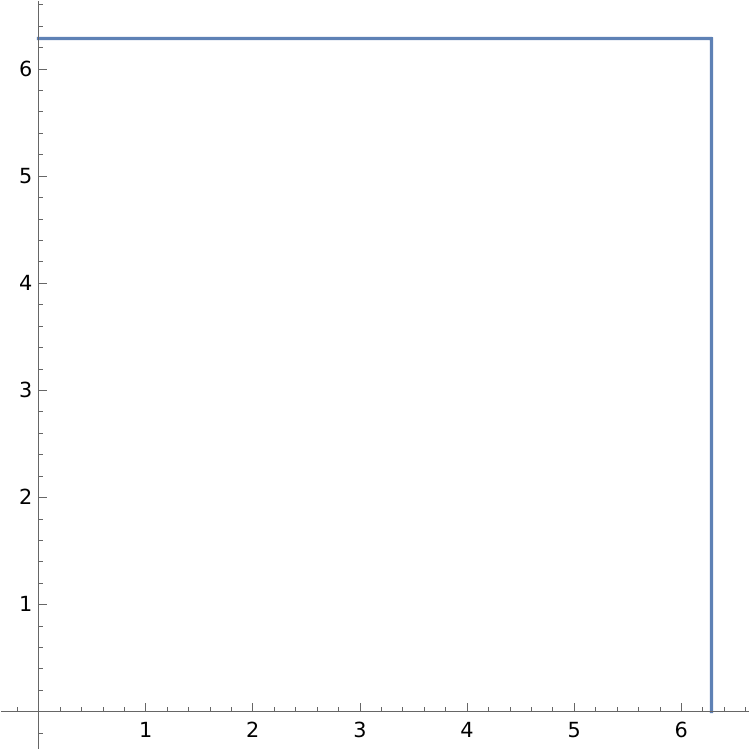}
\caption{$c=1$}
\end{subfigure}
\begin{subfigure}[t]{0.32\textwidth}
\centering
\includegraphics[scale=0.35]{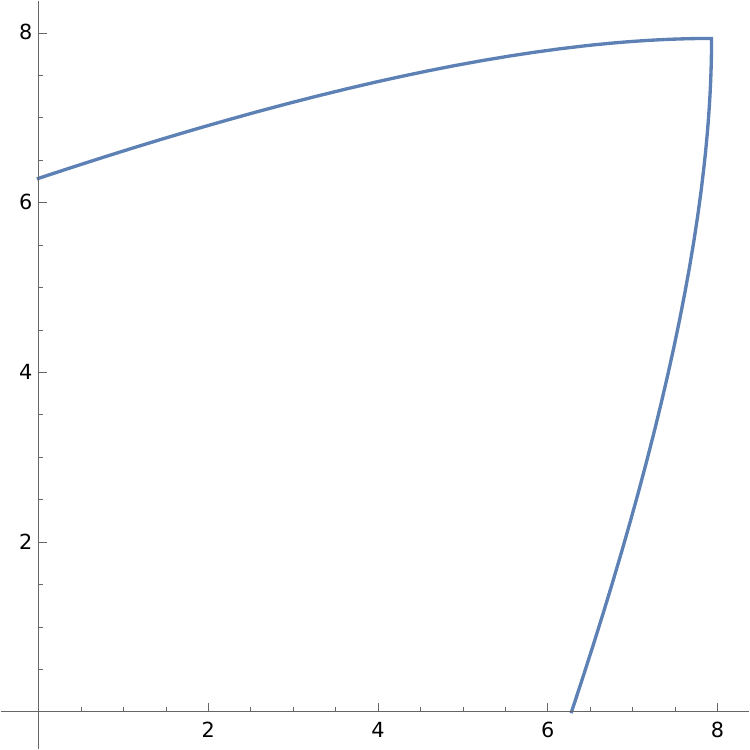}
\caption{$c=1.5$}
\end{subfigure}
\caption{The region $\Omega_c$ for different values of $c$.}\label{fig:cc}\label{fig: domains_c}
\end{figure}

\begin{rmk}
For $c=1$ in Proposition \ref{thm:toric} and the case of the round sphere in Proposition \ref{zollbidisk}, we recover the result in \cite{ferreira2021symplectic} stating that the disk cotangent bundle of the round sphere minus a point is symplectomorphic to the symplectic bidisk $\interior B^2(2\pi)\times B^2(2\pi)$.
\end{rmk}

The next step in the proof of Theorem \ref{thm:main} is the construction of the embedding of an appropriate ball into $\mathbb{X}_{\Omega_c}$. Let \begin{equation}\label{eq:w}w(c):=\left\{\begin{aligned} \alpha(c),&\text{ for }0<c<1/2,\\
2\pi,&\text{ for }1/2\le c \le 1,\\
\beta(c),&\text{ for }1<c<\beta^{-1}(4\pi),\\
4\pi,&\text{ for }c\ge \beta^{-1}(4\pi).
\end{aligned}\right.\end{equation}
\begin{prop}\label{prop:ballemb}
For every $c>0$, there exists a symplectic embedding \[\left(\interior B^4(w(c)),\omega_0\right)\hookrightarrow \left(\interior\mathbb{X}_{\Omega_c},\omega_0\right).\]
\end{prop}
As we will see in Section \ref{toricellip}, this embedding is just the inclusion for $c\le 1$. On the other hand, for $c>1$ we will need to use Cristofaro--Gardiner's highly nontrivial construction of a symplectic embedding from a concave into a weakly convex toric domain.

The last ingredient in the proof of Theorem \ref{thm:main} is the calculation of some ECH capacities of $(D^*\mathcal{E}(1,1,c),\omega_{can})$. We recall that ECH capacities are a sequence of symplectic capacities of four-dimensional symplectic manifolds. In particular,
\[(X_1,\omega_1)\hookrightarrow(X_2,\omega_2)\Rightarrow c_k(X_1,\omega_1)\le c_k(X_2,\omega_2)\;\text{ for every }k.\]
\begin{prop}\label{prop:capacities}
\
\begin{enumerate}[label=(\alph*)]
\item $c_3\left(D^*\mathcal{E}(1,1,c),\omega_{can}\right)=2w(c)$, for $0<c<\beta^{-1}(4\pi)$,
\item $c_1\left(D^*\mathcal{E}(1,1,c),\omega_{can}\right)=4\pi$, for $c\ge 1$.
\end{enumerate}
\end{prop}
Using Propositions \ref{thm:toric}, \ref{prop:ballemb} and \ref{prop:capacities}, we can prove our main theorem.

\begin{proof}[Proof of Theorem \ref{thm:main}]
It follows from Propositions \ref{prop:ballemb} and \ref{thm:toric} that for any $\varepsilon>0$, there exists a symplectic embedding 
\[B^4\left( (1-\varepsilon) w(c)\right)\hookrightarrow \left(D^* \mathcal{E}(1,1,c),\omega_{can}\right).\]
So $c_{Gr}\left(D^* \mathcal{E}(1,1,c),\omega_{can}\right)\ge (1-\varepsilon)w(c)$ for every $\varepsilon>0$ and hence \[c_{Gr}\left(D^* \mathcal{E}(1,1,c),\omega_{can}\right)\ge w(c).\]

Now suppose that $(B^4(a),\omega_0)\hookrightarrow \left(D^* \mathcal{E}(1,1,c),\omega_{can}\right)$. Recall that $c_1(B^4(a))=a$ and $c_3(B^4(a))=2a$. For $c<\beta^{-1}(4\pi)$, it follows from Proposition \ref{prop:capacities}(a) that
\[2a=c_3(B^4(a))\le c_3\left(D^*\mathcal{E}(1,1,c),\omega_{can}\right)=2w(c).\]
For $c\ge \beta^{-1}(4\pi)$, Proposition \ref{prop:capacities}(b) implies that
\[a=c_1(B^4(a))\le c_1\left(D^*\mathcal{E}(1,1,c),\omega_{can}\right)=4\pi=w(c).\]
So in either case $a\le w(c)$. Taking the supremum on $a$, we conclude that 
\[c_{Gr}\left(D^* \mathcal{E}(1,1,c),\omega_{can}\right)\le w(c).\]
\end{proof}

\noindent{\bf Structure of the paper:} In Section \ref{toricdomains} we use techniques from integrable systems to construct toric domains in the disk cotangent bundles of spheres of revolution. In particular, we prove Theorem \ref{toricspheres}, Propositions \ref{zollbidisk} and \ref{thm:toric}. In Section \ref{ECHcapacites} we recall what we need from ECH and we compute some ECH capacities, allowing us to prove Propositions \ref{prop:ballemb}, \ref{prop:capacities} and Theorem \ref{thm:grzoll}.
\newline

\noindent\textbf{Acknowledgments:}
We thank the anonymous referee/s for their invaluable comments and corrections. The second author is partially supported by grants from the Serrapilheira Institute, CNPq and FAPERJ.

\section{Toric domains in the disk cotangent bundles of spheres of revolution}\label{toricdomains}

\subsection{Integrability of geodesic flow}\label{embballs}
Let $S\subset \mathbb{R}^3$ be a sphere of revolution  as in (\ref{defS}). We endow $S$ with the metric induced by the ambient  Euclidean space $\R^3$ and denote it by $\langle \cdot, \cdot \rangle$. Further, we denote the induced norm by $\lVert \cdot \rVert$. It is well known that the meridians, namely, any intersection of $S$ with a plane containing the $z$-axis, are closed geodesics of $S$. The same hold for the equators, see e.g. \cite[Proposition 3.5.22]{klingenberg1995riemannian}.

The goal of this section is to show that if $S$ has a unique equator then $D^*S$ is the union of two copies of a toric domain and a measure zero set. We do this by using the integrability of the geodesic flow on spheres of revolution following ideas from \cite{zelditch1998inverse}. As a first consequence, one can study the geodesic flow of $S$ by studying the Reeb flow on the boundary of the corresponding toric domain. Further, we find symplectically embedded balls in the disk cotangent bundle $D^*S$ of $S$ under a twisting hypothesis using the Traynor trick described in \cite[$\S 5$]{traynor1995symplectic}.

Assume $S$ has a unique equator. We first consider the \emph{energy} function
\begin{eqnarray*}
H \colon T^*S &\to& \R \\
(q,p) &\mapsto& \lVert p \rVert^2.
\end{eqnarray*}
The Hamiltonian flow of $H$ is a reparametrization of the cogeodesic flow defined on $T^*S$. Here, by cogeodesic flow we mean the flow associated to the vector field which is dual to the geodesic vector field via the bundle isomorphism given by the metric
\begin{equation}\label{dualiso}
\begin{aligned}
\langle \cdot, \cdot \rangle^{\flat}  \colon TS &\to& T^*S \\
X &\mapsto& \langle X, \cdot \rangle.
\end{aligned}
\end{equation}

It is well known that the \emph{angular momentum} is an integral of motion for the Hamiltonian system defined by the energy. To define this function, consider the vector field $\partial_\theta = -q_2 \partial_{q_1} + q_1 \partial_{q_2}$ defined in Cartesian coordinates, i.e., the vector field on $S$ that generates the rotations around the $z$-axis. We can define the angular momentum $J$ by
\begin{eqnarray*}
J \colon T^*S &\to& \R \\
(q,p) &\mapsto& p(\partial_\theta).
\end{eqnarray*}

We will now use the fact that $(H,J)$ is an integrable system to find action-angle coordinates. Moreover, we obtain strict contactomorphisms as stated in the theorem below.

\begin{theorem}\label{contactomorphism}
Let $S$ be a sphere of revolution with only one equator and consider $F = (H,J)|_{D^*S} \colon D^*S \to \R^2$, where $H$ is the energy and $J$ is the angular momentum. Denote by $U_+,U_-$ the subsets corresponding to where $J$ is positive or negative, respectively. Then $U_+$ and $U_-$ are symplectomorphic to a toric domain. Moreover, the symplectomorphisms restrict to strict contactomorphisms on the boundaries $\partial U_+, \partial U_- \subset S^*S$.
\end{theorem}

\begin{proof}
The coordinates $(z, \theta)$ on $S$ induce cotangent coordinates $(z,\theta, p_z, p_\theta)$ on $T^*S$. In these coordinates, it follows from a simple calculation that
$$H(z,\theta,p_z,p_\theta) = \frac{p_z^2}{u'(z)^2 + 1} + \frac{p_\theta^2}{u(z)^2}, \quad J(z,\theta,p_z,p_\theta) = p_\theta.$$
In particular, $H$ and $J$ Poisson commute, i.e., $\{H,J\} = 0$. Since
$$J(q,p)^2 = \langle p^*, \partial_\theta \rangle^2 \leq \lVert p^* \rVert^2 \lVert \partial_\theta \rVert^2 = H(q,p) u(z)^2,$$
we have $\vert J(q,p) \vert \leq u(z_0) \sqrt{H(q,p)}$, for all $p \in T^*S$. Here we denote by $p^*$ the tangent vector dual to $p$ via the vector bundle isomorphism \eqref{dualiso} and $z_0 \in (a, b)$ is the unique critical point of $u$, where $u$ attains its maximum. The image of $F=(H,J)$ restricted to $D^*S$ is given by
\begin{equation}\label{eq:B}B:= \{(h,j)\in\R^2\mid   0 \leq h \leq 1 \ \text{and}\ \vert j \vert \leq u(z_0) \sqrt{h}\}.\end{equation}
The critical points of $F$ are the points $(q,p)$ such that one of the following conditions is satisfied:
\begin{itemize}
\item $p=0$ (the zero-section),
\item $z=z_0$ and $p_z=0$ (the equator with either orientation).
\end{itemize}
Hence, $B$ is the region bounded by the parabola $u(z_0)^2 h = j^2$ and by the line $h=1$, and the critical values of $F$ are exactly the points on this parabola. Now we apply the classical Arnold--Liouville Theorem in a subset of $D^*S$. Let
$$B_+ := \{(h,j)\in\R^2\mid  0 < h \leq 1 \ \text{and} \ 0< j < u(z_0)\sqrt{h}\}.$$
We define two families of circles $\gamma_1^{(h,j)}$ and $\gamma_2^{(h,j)}$ generating $H_1(F^{-1}(h,j); \Z)$ and depending smoothly on $(h,j)\in B_+$. Let \[\begin{aligned}\gamma_1^{(h,j)}&=\{(z,\theta,p_z,p_\theta)\in F^{-1}(h,j)\mid z=z_0,p_z>0\},\\ \gamma_2^{(h,j)}&=\{(z,\theta,p_z,p_\theta)\in F^{-1}(h,j)\mid \theta=0\}.\end{aligned}\]Then $\gamma_1^{(h,j)}$ and $\gamma_2^{(h,j)}$ are families of simple closed curves generating $H_1(F^{-1}(h,j); \Z)$. Let $\lambda=p_\theta\,d\theta+p_z\,dz$ be the tautological form on $T^* S$. In particular, $d\lambda=\omega_{can}$. The action coordinates are defined by
\begin{align}
I_1(h,j)&=\int_{\gamma_1^{(h,j)}} \lambda= 2 \pi j,\label{eq:int1}\\
I_2(h,j)&=\int_{\gamma_2^{(h,j)}} \lambda=2 \int_{z_{-}(h,j)}^{z_+(h,j)} \sqrt{\left(h - \frac{j^2}{u(z)^2}\right)(u'(z)^2 + 1)}\, dz,\label{eq:int2}
\end{align}
where $z_-(h,j)<z_+(h,j)$ are the two solutions to the equation $u(z)^2h - j^2 =0$. We observe that the map $\varphi=(I_1,I_2): B_+\to \R^2$ is a smooth embedding. Its image, which we denote by $\Omega$, is the open region in $\R^2_{>0}$ bounded by the curve parametrized by
\begin{equation}\label{curvec}
\gamma(j) = \left (2 \pi j, 2 \int_{z_{-}(1,j)}^{z_+(1,j)} \sqrt{\left(1 - \frac{j^2}{u(z)^2}\right)(u'(z)^2 + 1)}\, dz \right)
\end{equation}
for $j \in [0, u(z_0)]$.
Let $U_+=F^{-1}(B_+)$. It follows from the Arnold-Liouville theorem that there exists a symplectomorphism $\Phi \colon (U_+,\omega_{can}) \to (\mathbb{X}_\Omega,\omega_0)$, such that the diagram commutes
\[ \begin{tikzcd}
U_+\arrow{r}{\Phi} \arrow[swap]{d}{F} & \mathbb{X}_\Omega \arrow{d}{\mu} \\%
B_+  \arrow{r}{\varphi}& \Omega\
\end{tikzcd}
\]
where $\mu(z_1,z_2)=(\pi|z_1|^2,\pi|z_2|^2)$ is the standard moment map.

We now show that $\Phi$ can be chosen so that $\Phi^*\lambda_0=\lambda$, where $\lambda_0$ is the standard Liouville form on $\C^2$.
Following Arnold \cite{arnol2013mathematical}, we can find angle coordinates using generating functions as follows. First, note that our primitive $\lambda$ is closed on the Lagrangian torus $F^{-1}(h,j)$ since $d\lambda|_{F^{-1}(h,j)} = \omega_{can}|_{F^{-1}(h,j)} = 0$. For $(q,p)=(z,\theta,p_z,j)\in F^{-1}(h,j)$, we let $\gamma_{(q,p)}$ be a smooth path on $F^{-1}(B_+)$ from $(q_0,p_0):=(z_{-}(h,j),0,0,j)$ to $(q,p)$. We define the multivalued function $G$ by
\[G(q,p)=\int_{\gamma_{(q,p)}}\lambda.\]
The angle coordinates are heuristically defined by
\begin{equation}\label{eq:ang1}\phi_1=\frac{\partial G}{\partial I_1},\quad\phi_2=\frac{\partial G}{\partial I_2}.\end{equation}
More precisely, given a diffeomorphism $\psi:
(\R^2/\Z^2)\times B_+\to U_+$ such that $\psi(0,0,h,j)=(z_-(h,j),0,0,j)$, we consider a lift $\widetilde{\psi}:\R^2\times B_+\to U_+$ and we define a function
\[\widetilde{G}(t,(h,j))=\int_{(0,(h,j))}^{(t,(h,j))}\widetilde{\psi}^*\lambda,\]
where the integral above is over a path contained in $\R^2\times\{(h,j)\}$.
Then $\widetilde{G}$ can be seen as a well-defined lift of $G$. Note that $\widetilde{G}$ is independent of the choice of lift $\widetilde{\psi}$.
The angle coordinates are defined by
\begin{equation}\label{eq:ang2}\begin{aligned}\phi_1(q,p)&=\frac{\partial(\widetilde{G}\circ (\text{id}\times\varphi^{-1}))}{\partial I_1}(t,\varphi(h,j)),\\
\phi_2(q,p)&=\frac{\partial(\widetilde{G}\circ (\text{id}\times\varphi^{-1}))}{\partial I_2}(t,\varphi(h,j))
\end{aligned}
\end{equation}
where $\widetilde{\psi}(t,(h,j))=(q,p)$. We observe that the partial derivatives above are independent of the choice of the preimage $\widetilde{\psi}^{-1}(q,p)$ so they can be seen as partial derivatives of $G$. So \eqref{eq:ang1} can be seen as a simplified expression for \eqref{eq:ang2}.

As usual, we define
\[\Phi\colon U_+\to \mathbb{X}_\Omega,\qquad (q,p)\mapsto\left(\sqrt{\frac{I_1(q,p)}{\pi}}e^{2\pi \phi_1(q,p)},\sqrt{\frac{I_2(q,p)}{\pi}}e^{2\pi \phi_2(q,p)}\right)\]
It is well-known that $\Phi$ is a symplectomorphism, see \cite{arnol2013mathematical}. We now prove that $\Phi^*\lambda_0=\lambda$. This fact follows from a more general result, namely a version of Arnold--Liouville theorem for $1$-homogeneous\footnote{Note that $\sqrt{H}$ and $J$ are $1$-homogeneous.} integrable systems in Liouville domains, see \cite{colombo2023liouville}. For completeness, we provide a proof in our specific case. Taking polar coordinates in each complex variable of $\C^2$, we can write 
\[\lambda_0=\frac{1}{2}\sum_{i=1}^2 r_i^2\, d\theta_i.\]
So \[\Phi^*\lambda_0=I_1\,d\phi_1+I_2\,d\phi_2.\]
\newline
\noindent\textbf{Claim}: $I_1\,d\phi_1+I_2\,d\phi_2=\lambda$.
\begin{proof}[Proof of the Claim] 
Let $(q,p)=(z,\theta,p_z,p_\theta)$ such that $p_z\neq 0$ and let $(h,j)=F(q,p)$. We choose a preimage $(t,(h,j))\in\widetilde{\psi}^{-1}(q,p)$ such that the $p_z$-coordinate of $\widetilde{\psi}(s,F(q,p))$ does not vanish for $0<s\le t$.
Let
\begin{equation}\label{eq:P}
P(w,h,j):=\sqrt{\left(h - \frac{j^2}{u(w)^2}\right)(u'(w)^2 + 1)}.
\end{equation} So $\widetilde{G}(t,(h,j))=\widetilde{G}(z,\theta,h,j)$ is given by
 \begin{equation}\label{eq:G}
\widetilde{G}(z,\theta,h,j)=
\sigma\int_{z_-(h,j)}^z P(w,h,j)\, dw + \theta j,\end{equation}
where $\sigma\in\{-1,1\}$ is the sign of $p_z$. We note that \eqref{eq:G} holds in a neighborhood of $(t,(h,j))$. By slightly abusing notation, we write $(h,j):=\varphi^{-1}(I_1,I_2)$.
It follows from \eqref{eq:int1}, \eqref{eq:int2} and \eqref{eq:P} that
\begin{equation}\label{eq:dhj}
\begin{aligned}
\frac{\partial h}{\partial I_1}&=-\frac{\int_{z_-(h,j)}^{z_+(h,j)} \partial_j P(w,h,j)\,dw}{2\pi\int_{z_-(h,j)}^{z_+(h,j)} \partial_h P(w,h,j)\,dw}, &\frac{\partial j}{\partial I_1}&=\frac{1}{2\pi},\\
\frac{\partial h}{\partial I_2}&=\frac{1}{2\int_{z_-(h,j)}^{z_+(h,j)} \partial_h P(w,h,j)\,dw},& \frac{\partial j}{\partial I_2}&=0.
\end{aligned}
\end{equation}
Using \eqref{eq:int1}, \eqref{eq:int2}, \eqref{eq:ang2}, \eqref{eq:P} and \eqref{eq:dhj}, we compute the angle coordinates:
\begin{equation}\label{eq:angle}
\begin{aligned}
\phi_1(z,\theta,h,j)&=-\frac{\sigma\int_{z_-(h,j)}^{z} \partial_h P(w,h,j)\,dw\cdot \int_{z_-(h,j)}^{z_+(h,j)} \partial_j P(w,h,j)\,dw}{2\pi\int_{z_-(h,j)}^{z_+(h,j)} \partial_h P(w,h,j)\,dw}\\&+\frac{\sigma}{2\pi}\int_{z_-(h,j)}^{z} \partial_j P(w,h,j)\,dw+\frac{\theta}{2\pi},\\
\phi_2(z,\theta,h,j)&=\frac{\sigma\int_{z_-(h,j)}^{z} \partial_h P(w,h,j)\,dw}{2\int_{z_-(h,j)}^{z_+(h,j)} \partial_h P(w,h,j)\,dw}.
\end{aligned}
\end{equation}
It follows from \eqref{eq:P} that
\begin{equation}\label{eq:dp}
P=j\partial_jP+2h\partial_h P.
\end{equation}
From \eqref{eq:angle} and \eqref{eq:dp}, we obtain
\begin{equation}\label{eq:IG}\begin{aligned}
I_1(h,j)\phi_1(z,\theta,h,j)&= 2\pi j\phi_1(z,\theta,h,j)\\
&=-I_2(h,j)\phi_2(z,\theta,h,j)+\sigma\int_{z_-(h,j)}^z P(w,h,j)\,dw+\theta j.\end{aligned}
\end{equation}
It follows from \eqref{eq:G} and \eqref{eq:IG} that
\begin{equation}\label{eq:iphi}I_1\phi_1+I_2\phi_2=\widetilde{G}.\end{equation} Differentiating this equation and using \eqref{eq:ang2} and \eqref{eq:G} we obtain
\[\begin{aligned}
I_1\,d\phi_1+I_2\,d\phi_2+\phi_1\,dI_1+\phi_2\,dI_2&=\sigma P \,dz+j\,d\theta+\partial_{I_1}\widetilde{G} \,dI_1+\partial_{I_2} \widetilde{G}\, dI_2\\
&=p_z\,dz+p_\theta\,d\theta+\phi_1\, dI_1+\phi_2\,dI_2.
\end{aligned}\]
Therefore $I_1\,d\phi_1+I_2\,d\phi_2=\lambda$ proving the claim whenever $p_z\neq 0$. By continuity, this equation also holds when $p_z=0$.
\end{proof}

We conclude that the symplectomorphism $\Phi$ restricts to a strict contactomorphism $\Phi|_{\partial U_+}$ between $\partial U_+$ and $\partial \mathbb{X}_{\Omega}$. Analogously, one can reproduce the argument above to the symmetric set $U_- := F^{-1}(B_-)$, where
$$B_- := \{(h,j)\mid 0 < h \leq 1 \ \text{and} \ -u(z_0)\sqrt{h} < j <0 \},$$
and conclude the same thing, i.e., $U_-$ is symplectomorphic to $\mathbb{X}_{\Omega}$ and the symplectomorphism restricts to a strict contactomorphism between the boundaries. In particular, $U_+$ and $U_-$ are symplectomorphic.
\end{proof}
\begin{rmk}\label{rmk:angle}
Since $I_1(z,\theta,p_z,p_\theta)=2\pi p_\theta$, it follows that $X_{I_1}= 2\pi\partial_{\theta}$. Using \eqref{eq:G} and \eqref{eq:iphi}, we obtain\[\begin{aligned}\phi_1(z,\theta,h,j)&=\frac{\theta}{2\pi},\\
\phi_2(z,\theta,h,j)&=\frac{\sigma\int_{z_-(h,j)}^{z} P(w,h,j)\,dw}{2\int_{z_-(h,j)}^{z_+(h,j)} P(w,h,j)\,dw}.\end{aligned}\]
Hence, the circles generated by $\partial_{\phi_1}$ are the lifts of the parallels between $z_-(h,j)$ and $z_{+}(h,j)$, i.e., the points $(z,\theta,p_z,p_\theta)$, where $\theta\in[0,2\pi)$ and the other coordinates are fixed. Moreover, the circles generated by $\partial_{\phi_2}$ are lifts of curves that travel along a meridian crossing the equator twice, i.e., the set of points $(z,\theta,p_z,p_\theta)$ in $F^{-1}(h,j)$ with a fixed $\theta$.
\end{rmk}
\begin{rmk}\label{rmkeq}
Using theorems of Eliasson \cite{eliasson1984hamiltonian,eliasson1990normal} as done in \cite{Ramos2017SymplecticEA} and \cite{ostrover2021symplectic}, we can extend the symplectomorphisms above to the sets
\[\begin{aligned}\{(h,j)\in\R^2\mid 0<h\le 1\text{ and }0<j\le u(z_0)\sqrt{h}\},\\
\{(h,j)\in\R^2\mid 0<h\le 1\text{ and }- u(z_0)\sqrt{h}\le j<0\}.\end{aligned}\] By continuity, we also have $\Phi^* \lambda_0 = \lambda$ for these two sets (including $h=1$). But we cannot extend them for $j=0$, since $F^{-1}(h,0)$ is a torus for $h\neq 0$ and not an $S^1$. In the next section, we will explain how to slightly modify the action-angle coordinates to obtain a symplectomorphism of the complement of a fiber in $F^{-1}(B)$ onto a toric domain.
\end{rmk}

Let $\mathfrak{m} = \min \{2\pi u(z_0),L\}$, with $L$ denoting the length of any meridian on $S$. In other words, $\mathfrak{m}$ is the minimum between the length of the equator and the length of a meridian on $S$. It is readily verified from equation \eqref{curvec} that the curve $\gamma$ has the same trace as the graph of a function $f_S$ defined by
\begin{eqnarray*}
f_S\colon [0,2\pi u(z_0)] &\to& \R_{\geq 0}
\\ l &\mapsto& 2 \int_{z_{-}(1,l/2\pi)}^{z_+(1,l/2\pi)} \sqrt{\left(1 - \frac{l^2}{4\pi^2 u(z)^2}\right)(u'(z)^2 + 1)}\, dz,
\end{eqnarray*}
Hence, if the graph of $f_S$ is above the line $x+y = \mathfrak{m}$, the open triangle
$$\bigtriangleup (\mathfrak{m}) = \{(x,y) \in \R_{>0}^2 \mid x+y < \mathfrak{m}\},$$
is contained in $\Omega$. Moreover, since $f_S$ satisfies $f_S(0) = L$ and $f_S(2\pi u(z_0)) = 0$, the latter holds whenever $f_S$ is a concave function.  Therefore, one can apply the Traynor trick \cite[$\S 5$]{traynor1995symplectic} in the version stated in \cite[Proposition 4.3]{gutt2022examples} to obtain the following result.

\begin{prop}\label{ball}
Suppose that $S$ has only one equator and that the function $f_S$ is a concave function. Let $\mathfrak{m} = \min \{2\pi u(s_0),L\}$. Then, there exists a symplectic embedding
$$((1-\varepsilon)B^4(\mathfrak{m}),\omega_0) \hookrightarrow (D^*S,\omega_{can}),$$
for every $\varepsilon \in (0,1)$. In particular, $c_{Gr}((D^*S,\omega_{can}) \geq \mathfrak{m}$.
\end{prop}

In fact, given $\varepsilon \in (0,1)$ an affine copy of the triangle $(1-\varepsilon)\overline{\bigtriangleup(\mathfrak{m})}$ is contained in $\Omega$, and hence, the Traynor trick ensures the existence of such a symplectic embedding. 

\begin{rmk}
\begin{enumerate}[label=(\roman*)]
\item  The discussion above can be adapted to the case where $S$ has more than one equator. In this case, it is possible to conclude that the Gromov width of $(D^*S,\omega_{can})$ is at least the minimum between the length of the smallest equator of $S$ and the length of a meridian.
\item The derivative of the function $f_S$ is directly related to the equatorial first return angle defined by Zelditch, see \cite{zelditch1998inverse}. In this case, one can prove that the convexity of $f_S$ is related to a twisting condition on the (non-linear) Poincaré map for the equator.
\end{enumerate}
\end{rmk}

\begin{ex}[Round sphere]
The round sphere $S^2 \subset \R^3$ is described by the function $u(z) = \sqrt{1-z^2}$, for $z \in [-1,1]$. In this case, we have
\begin{align*}
f_S(l) &= 2 \int_{-\sqrt{1- (l/2\pi)^2}}^{\sqrt{1-(l/2\pi)^2}} \sqrt{\left(1- \frac{l^2}{4\pi^2(1-z^2)}\right)\left(\frac{1}{1-z^2}\right)}dz \\ &= \frac{2}{\pi}\int_0^{\sqrt{1- (l/2\pi)^2}} \sqrt{\frac{4\pi^2(1-z^2)-l^2}{(1-z^2)^2}}dz \\ &= 4 \int_0^{\pi/2} \frac{\cos^2\theta(4\pi^2-l^2)}{l^2+\cos^2\theta(4\pi^2-l^2)} d\theta \\ &= 2\pi -4l^2\left(\int_0^{2\pi}\frac{d\theta}{l^2+\cos^2\theta(4\pi^2-l^2)}\right) \\ &= 2\pi - 4l^2\left[\frac{\arctan\left(\frac{l \tan \theta}{2\pi}\right)}{2\pi l}\right]_{\theta=0}^{\theta=\frac{\pi}{2}} \\ &=2\pi-l,
\end{align*}
for $l \in [0,2\pi]$, where we make the substitution $\theta = \arcsin(z/\sqrt{1-(l/2\pi)^2})$. Hence, Proposition \ref{ball} ensures the existence of a symplectic embedding
$$((1-\varepsilon)B^4(2\pi),\omega_0) \hookrightarrow (D^*S,\omega_{can}),$$
for every $\varepsilon \in (0,1)$.
\end{ex}
In fact, the first two named authors showed in \cite{ferreira2021symplectic} that $(\interior (B^4(2\mathfrak{\pi})),\omega_0)$ is symplectomorphic to $(\interior D^*\Sigma,\omega_{can})$, where $\Sigma$ is any open hemisphere in the sphere $S^2$.

\subsection{Toric domains hidden in spheres of revolution}\label{sec:intsystellip}
In this section we prove Theorem \ref{toricspheres}. To do that we will use a perturbed version of the integrable system appearing in the proof of Theorem \ref{contactomorphism}.

\begin{proof}[Proof of Theorem \ref{toricspheres}]
We use the same notation as in the proof of Theorem \ref{contactomorphism}. Let $0<\epsilon<b-z_0$, with $b$ as in (\ref{defS}), and let $U_\epsilon\in C^\infty([0,b-\epsilon))$ be a non-decreasing smooth function supported in $(z_0,b-\epsilon)$, such that $U_\epsilon(z)=\frac{\epsilon}{b-\epsilon-z}$ for $z$ sufficiently close to $b-\epsilon$. We also assume that $U_\epsilon''(z)\ge 0$ and that the family $U_\epsilon$ is increasing in $\epsilon$. For $(q,p)=(z,\theta,p_z,p_\theta)\in D^* S$ such that $z<b-\epsilon$, we define
\[H^{\epsilon}(z,\theta,p_z,p_\theta)=H(z,\theta,p_z,p_\theta)+ U_\epsilon(z).\]
Note that $H^{\epsilon}$ does not depend on $\theta$ and so it still Poisson commutes with $J$.
Now let $F^{\varepsilon}=(H^{\varepsilon},J)$ whose domain is
\begin{equation*}\label{set}
W^{\varepsilon}:=\{(z,\theta,p_z,p_\theta)\in D^*S\mid H^{\epsilon}(z,\theta,p_z,p_\theta)\le 1\}.
\end{equation*}
Since $U_\varepsilon$ is nonnegative and supported in $(z_0,b-\varepsilon)$, one can check that $F^{\varepsilon}(W^{\varepsilon})=B$, as defined in \eqref{eq:B}. Further, it follows from the assumptions on $U_\epsilon$ that $(h,j)$ is a critical value of $F^{\varepsilon}$ if, and only if $|j|=u(z_0)\sqrt{h}$. Therefore, we can apply the Arnold-Liouville theorem in  
\[\interior(B)=\{(h,j)\in\R^2\mid 0<h<1,\,|j|<u(z_0)\sqrt{h}\}.\]
So, there exists a domain $\Omega^{\epsilon}$ which we will prove later on is a subset of $\R_{>0}^2$, a diffeomorphism $\varphi^{\epsilon}:\interior(B)\to\Omega^{\epsilon} $ and a symplectomorphism $\Phi^{\epsilon}:(F^{\varepsilon})^{-1}(\interior(B))\to \mathbb{X}_{\Omega^{\epsilon}}$ such that the following diagram commutes.
\[ \begin{tikzcd}
(F^{\varepsilon})^{-1}(\interior(B)) \arrow{r}{\Phi^{\varepsilon}} \arrow[swap]{d}{F^{\varepsilon}} & \mathbb{X}_{\Omega^{\varepsilon}} \arrow{d}{\mu} \\%
\interior(B)  \arrow{r}{\varphi^{\varepsilon}}& \Omega^{\varepsilon}.
\end{tikzcd}
\]
As in Section \ref{embballs}, we can describe the diffeomorphism $\varphi^\varepsilon$ integrating a primitive $\lambda$ of $\omega_{can}$ over two smooth families of loops $\{\zeta_1^{(h,j)},\zeta_2^{(h,j)}\}\subset(F^{\varepsilon})^{-1}(h,j)$ that generate the homology of this torus. The analogous of the families defined in Section \ref{embballs} do not work, because it is not possible to extend them to the locus where $j=0$ continuously. Instead, we proceed as in \cite{ferreira2021symplectic} and \cite{ostrover2021symplectic}. For $(h,j)\in \interior(B)$, let $z^{\varepsilon}_{+}(h,j)$ and $z^{\varepsilon}_{-}(h,j)$ be the largest and smallest solutions of \[\frac{j^2}{u(z)^2}+ U_\epsilon(z)=h,\]respectively. Define $\sigma_0^{(h,j)}$ to be the curve following the flow of $X_{H^{\varepsilon}}$, starting at $(z^{\varepsilon}_{+}(h,j),0,0,j)$ until the next point $(z,\theta,p_z,p_\theta)$ such that $z=z^{\varepsilon}_{+}(h,j)$. It follows that $(z,\theta_,p_z,p_\theta)=(z^{\varepsilon}_{+}(h,j),\theta,0,j)$. From the definition of $X_{H^{\varepsilon}}$ and the basic theory of ordinary differential equations, $\theta=\theta^{\varepsilon}(h,j)\in \R/(2\pi\Z)$ is continuous in $(h,j)$. Moreover, it is simple to check that when $j$ goes to zero, $\theta^{\varepsilon}(h,j)$ converges to $\pi$, and hence $\theta^{\varepsilon}(h,0)=\pi$. We now define $\sigma_1^{(h,j)}$ and $\sigma_2^{(h,j)}$ to be curves connecting $(z^{\varepsilon}_-(h,j),\theta^{\varepsilon}(h,j),0,j)$ and $(z^{\varepsilon}_{+}(h,j),0,0,j)$ with the following property. We take $\sigma_1^{(h,0)}$ and $\sigma_2^{(h,0)}$ to be the curves parametrized by \[\sigma_i^{(h,0)}(t)=(z^{\varepsilon}_{+}(h,0),\pi+(-1)^i \pi t,0,j), \quad t\in[0,1].\] It follows from the theory of covering spaces that $\sigma_1^{(h,0)}$ and $\sigma_2^{(h,0)}$ extend to unique, up to homotopy, families of curves $\sigma_1^{(h,j)}$ and $\sigma_2^{(h,j)}$. For $i=1,2$, we take $\zeta_i^{(h,j)}$ to be a smoothening of the composition of $\sigma_0^{(h,j)}$ with $\sigma_i^{(h,j)}$. For $j\neq 0$ we have
\begin{equation}\label{tautform}
    \int_{\zeta_i^{(h,j)}}\lambda=\int_{\zeta_i^{(h,j)}}p_z\,dz+\int_{\zeta_i^{(h,j)}}p_{\theta}\,d\theta=\int_{\sigma_0}p_z\,dz+\int_{\zeta_i^{(h,j)}}p_{\theta}\,d\theta.
\end{equation}
From a simple calculation, we obtain
\begin{equation}
    \int_{\sigma_0}p_z\,dz=2\int_{z_{-}^{\varepsilon}(h,j)}^{{z_{+}^{\varepsilon}(h,j)}}\sqrt{\left(h-\frac{j^2}{u(z)^2}- U_\epsilon(z)\right)\left(1+u'(z)^2\right)}\,dz.\label{eq:pz}
\end{equation}
Now we claim that
\begin{equation}
   \Theta_i(j):= \int_{\zeta_i^{(h,j)}}p_{\theta}\,d\theta=\left\{\begin{aligned}
        2\pi j&, \text{if } i=2,j>0\\
    0&, \text{if } i=2,j<0 \text{ or } i=1,j>0,\\
    -2\pi j&, \text{if } i=1,j<0,
    \end{aligned}\right.
    \label{integrals}
    \end{equation}
\begin{proof}[Proof of the Claim]    
For every $j\neq 0$, the curve $\zeta_i^{(h,j)}$ never goes through the point $P_S=(0,0,a)$, i.e., the projection of $\zeta_i^{(h,j)}$ to the $xy$-plane does not go through the origin. So the value of $\displaystyle\int_{\zeta_i^{(h,j)}} d\theta$ is constant for $j>0$ and for $j<0$. We fix $h>0$ and we will compute these two values by taking the limit of this integral as $j\to 0^+$ and as $j\to 0^-$. First note that
\begin{equation}\label{eq:h0} \int_{\sigma_1^{(h,0)}}d\theta=-\pi \text{ and }\int_{\sigma_2^{(h,0)}}d\theta=\pi.\end{equation}
Now we compute
\[\lim_{j\to 0^{\pm}}\int_{\sigma_0^{(h,j)}}d\theta.\] From the basic theory of differential equations and differential forms, we obtain
\begin{equation}\label{eq:intsigma0}\int_{\sigma_0^{(h,j)}}d\theta=2\int_{z_-^\epsilon{(h,j)}}^{z_+^\epsilon{(h,j)}}\frac{j\sqrt{1+u'(z)^2}}{u(z)^2 \sqrt{h-\frac{j^2}{u(z)^2}- U_\epsilon(z)}}\,dz.\end{equation}
We recall that $u(z_+^{\epsilon}(h,j))<b-\epsilon$, that $u(z_-^{\epsilon}(h,j))=|j|/\sqrt{h}$ and that $u$ is strictly increasing in the interval $(a,z_0)$. We choose $\bar{z}_-$ and $\bar{z}_+$ such that $a<\bar{z}_-<z_0<\bar{z}_+<b-\epsilon$ and that $U_\epsilon'(\bar{z}_+)>0$. In particular, for $|j|$ sufficiently small, $z_-^{\epsilon}(h,j)<\bar{z}_-<\bar{z}_+<z_+^\epsilon(h,j)$. We now break the interval of integration of the integral in \eqref{eq:intsigma0} in three parts: $[z_-^{\epsilon}(h,j),\bar{z}_-]$, $[\bar{z}_-,\bar{z}_+]$ and $[\bar{z}_+,z_+^\epsilon(h,j)]$.

In $[\bar{z}_+,z_+^\epsilon(h,j)]$ the function $z\mapsto h-\frac{j^2}{u(z)^2}-U_\epsilon(z)$ only vanishes at $z_+^\epsilon(h,j)$. Using the mean value theorem and the fact that $U_\epsilon''(z)\ge 0$, we obtain
\begin{equation}
\begin{aligned}
\Bigg|\int_{\bar{z}_+}^{z_+^\epsilon(h,j)}\frac{j\sqrt{1+u'(z)^2}}{u(z)^2 \sqrt{h-\frac{j^2}{u(z)^2}- U_\epsilon(z)}}\,dz\Bigg|&\le |j|\int_{\bar{z}_+}^{z_+^\epsilon(h,j)}\frac{\sqrt{1+u'(z)^2}\,dz}{u(z)^2 \sqrt{h-\frac{j^2}{u(z)^2}- U_\epsilon(z)}}\\
&\le |j|\int_{\bar{z}_+}^{z_+^\epsilon(h,j)}\frac{\sqrt{1+u'(b-\epsilon)^2}\,dz}{u(b-\epsilon)^2 \sqrt{U_\epsilon'(\bar{z}_+)(z_+^\epsilon(h,j)-z)}}\\
&\le C |j|.
\end{aligned}
\end{equation}
So \begin{equation}\label{eq:thirdpart}
\lim_{j\to { 0^{\pm}}}\int_{\bar{z}_+}^{z_+^\epsilon(h,j)}\frac{j\sqrt{1+u'(z)^2}}{u(z)^2 \sqrt{h-\frac{j^2}{u(z)^2}- U_\epsilon(z)}}\,dz =0.
\end{equation}
Now in $[\bar{z}_-,\bar{z}_+]$, the integrand of \eqref{eq:intsigma0} is continuous and is monotonic in $j$. By the monotone convergence theorem, it follows that
\begin{equation}\label{eq:secondpart}
\lim_{j\to { 0^{\pm}}}\int_{\bar{z}_-}^{\bar{z}_+}\frac{j\sqrt{1+u'(z)^2}}{u(z)^2 \sqrt{h-\frac{j^2}{u(z)^2}- U_\epsilon(z)}}\,dz =0.
\end{equation}

We now fix $\delta>0$. Since $u'(z)\to \infty$ as $z\to a$, we can choose $\rho\in(a,\bar{z}_-)$ such that $\bigg|\frac{\sqrt{1+u'(z)^2}}{u'(z)}- 1\bigg|<\delta$ for all $z\in(a,\rho)$. If $|j|$ is sufficiently small, then $z_-^\epsilon(h,j)<\rho$.
In $[\rho,\bar{z}_-]$, we apply the same argument as in $[\bar{z}_-,\bar{z}_+]$ to conclude that
\begin{equation}\label{eq:firstpartb}
\lim_{j\to {0^{\pm}}}\int_{\rho}^{\bar{z}_-}\frac{j\sqrt{1+u'(z)^2}}{u(z)^2 \sqrt{h-\frac{j^2}{u(z)^2}- U_\epsilon(z)}}\,dz =0.
\end{equation}
Finally, in order to compute the integral in $[z_-^\epsilon(h,j),\rho]$, we first use the substitution $\sin \theta=\frac{|j|}{u(z)\sqrt{h}}$ to compute the following integral:
\begin{equation}
\begin{aligned}
\label{eq:firstint}
\int_{z_-^\epsilon(h,j)}^{\rho}\frac{ju'(z)\,dz}{u(z)^2\sqrt{h-\frac{j^2}{u(z)^2}}}&=-\mathrm{sgn}(j)\int_{\pi/2}^{\arcsin\frac{|j|}{u(\rho)\sqrt{h}}} d\theta\\&=\mathrm{sgn}(j)\left(\frac{\pi}{2}-\arcsin\frac{|j|}{u(\rho)\sqrt{h}}\right).\end{aligned}
\end{equation}
Now we estimate the difference between the original integral in $[z_-^\epsilon(h,j),\rho]$ and the integral in \eqref{eq:firstint}. For this calculation, recall that $U_\epsilon(z)=0$ for all $z<z_0$.
\begin{equation}\label{eq:estimate}
\begin{aligned}
&\Bigg| \int_{z_-^\epsilon(h,j)}^{\rho}\frac{j\sqrt{1+u'(z)^2}}{u(z)^2 \sqrt{h-\frac{j^2}{u(z)^2}- U_\epsilon(z)}}\,dz -\int_{z_-^\epsilon(h,j)}^{\rho}\frac{ju'(z)\,dz}{u(z)^2\sqrt{h-\frac{j^2}{u(z)^2}}}\Bigg|\\&=\Bigg|\int_{z_-^\epsilon(h,j)}^{\rho}\frac{ju'(z)}{u(z)^2\sqrt{h-\frac{j^2}{u(z)^2}}}\left(\frac{\sqrt{1+u'(z)^2}}{u'(z)}-1\right)\,dz\Bigg|\\
&\le \int_{z_-^\epsilon(h,j)}^{\rho}\frac{ju'(z)}{u(z)^2\sqrt{h-\frac{j^2}{u(z)^2}}}\left|\frac{\sqrt{1+u'(z)^2}}{u'(z)}-1\right|\,dz\\
&\le \delta\int_{z_-^\epsilon(h,j)}^{\rho}\frac{|j|u'(z)\,dz}{u(z)^2\sqrt{h-\frac{j^2}{u(z)^2}}}\\
&=\delta\left(\frac{\pi}{2}-\arcsin\frac{|j|}{u(\rho)\sqrt{h}}\right)<\frac{\delta\pi}{2}
\end{aligned}
\end{equation}
It follows from \eqref{eq:firstint} and \eqref{eq:estimate} that
\begin{equation*}
\begin{aligned}
&\Bigg|\int_{z_-^{\epsilon}(h,j)}^{\bar{z}_-}\frac{j\sqrt{1+u'(z)^2}}{u(z)^2 \sqrt{h-\frac{j^2}{u(z)^2}- U_\epsilon(z)}}\,dz-\mathrm{sgn}(j)\left(\frac{\pi}{2}-\arcsin\frac{|j|}{u(\rho)\sqrt{h}}\right)\Bigg|\\
&=
\Bigg|\int_{z_-^{\epsilon}(h,j)}^{\bar{z}_-}\frac{j\sqrt{1+u'(z)^2}}{u(z)^2 \sqrt{h-\frac{j^2}{u(z)^2}- U_\epsilon(z)}}\,dz-\int_{z_-^\epsilon(h,j)}^{\rho}\frac{ju'(z)\,dz}{u(z)^2\sqrt{h-\frac{j^2}{u(z)^2}}}\Bigg|\\
&<\frac{\delta \pi}{2}+\Bigg|\int_{\rho}^{\bar{z}_-}\frac{j\sqrt{1+u'(z)^2}}{u(z)^2 \sqrt{h-\frac{j^2}{u(z)^2}- U_\epsilon(z)}}\,dz\Bigg|.
\end{aligned}
\end{equation*}
Taking the limit as $j\to 0^{\pm}$ and using \eqref{eq:firstpartb} we obtain
\begin{equation}
\Bigg|\lim_{j\to 0^{\pm}}\int_{z_-^{\epsilon}(h,j)}^{\bar{z}_-}\frac{j\sqrt{1+u'(z)^2}}{u(z)^2 \sqrt{h-\frac{j^2}{u(z)^2}- U_\epsilon(z)}}\,dz-\left(\pm\frac{\pi}{2}\right)\Bigg|<\frac{\delta \pi}{2}.
\end{equation}
Since $\delta>0$ can be taken to be arbitrarily small, it follows that
\begin{equation}\label{eq:firstpart}
\lim_{j\to 0^{\pm}}\int_{z_-^{\epsilon}(h,j)}^{\bar{z}_-}\frac{j\sqrt{1+u'(z)^2}}{u(z)^2 \sqrt{h-\frac{j^2}{u(z)^2}- U_\epsilon(z)}}\,dz=\pm\frac{\pi}{2}.
\end{equation}
Combining \eqref{eq:h0}, \eqref{eq:intsigma0}, \eqref{eq:thirdpart}, \eqref{eq:secondpart} and \eqref{eq:firstpart}, we obtain
\[\begin{aligned}\lim_{j\to 0^{+}}\int_{\zeta_i^{(h,j)}}d\theta&=\left\{\begin{aligned}0,&\text{ if }i=1,\\
2\pi,&\text{ if }i=2,\end{aligned}\right.\\
\lim_{j\to 0^{-}}\int_{\zeta_i^{(h,j)}}d\theta&=\left\{\begin{aligned}-2\pi,&\text{ if }i=1,\\
0,&\text{ if }i=2,\end{aligned}\right.\\
\end{aligned}\]
The claim now follows.
\end{proof}

We now note that \eqref{integrals} is non-negative and \eqref{eq:pz} is positive. It follows that $\Omega^{\varepsilon}=\varphi^{\varepsilon}(\interior(B))$ is contained in $\R^2_{> 0}$.
Let $\overline{\Omega}^{\varepsilon}$ be the closure of $\Omega^{\varepsilon}$ in $\R^2_{\ge 0}$. We can again use theorems of Eliasson \cite{eliasson1984hamiltonian} to extend $\Phi^{\varepsilon}$ to a symplectomorphism $W^{\varepsilon}\cong \mathbb{X}_{\overline{\Omega}^{\varepsilon}}$.\\
Now notice that 
$$D^*(S\setminus\{(0,0,b\})=\bigcup_{\varepsilon > 0}W^{\varepsilon}.$$
A straightforward computation shows that 
$$\varepsilon_1<\varepsilon_2 \Rightarrow \overline{\Omega}^{\varepsilon_2}\subset \overline{\Omega}^{\varepsilon_1}.$$
Let
$$\Omega=\bigcup_{\varepsilon > 0}\overline{\Omega}^{\varepsilon}.$$
By the same argument as in the proof of \cite[Theorem 1.3]{ferreira2021symplectic}, we obtain a symplectomorphism 
$$D^*(S\setminus\{(0,0,b\})\cong \bigcup_{\varepsilon >0}\mathbb{X}_{\Omega^{\varepsilon}}=\mathbb{X}_{\Omega}. $$
Finally, we compute the boundary of $\Omega$. The boundary of the region $\Omega^{\varepsilon}$ is the curve parametrized by $\varphi^{\varepsilon}(1,j)$. Therefore $\Omega$ is the relatively open set bounded by the curve
$$(\rho_1(j),\rho_2(j))=\lim_{\varepsilon\rightarrow 0} \varphi^{\varepsilon}(1,j).$$
As $\varepsilon\rightarrow 0$ the domain of parametrization is $[-u(z_0),u(z_0)]$. This is a continuous curve so it suffices to compute $\rho_i(j)$ for $0<j^2<u(z_0)^2$. From (\ref{tautform}) and (\ref{integrals}) we obtain 
\begin{equation}\label{parametrization}
    \rho_i(j)=2\int_{z_-(1,j)}^{z_+(1,j)}\sqrt{\left(1-\frac{j^2}{u(z)^2}\right)\left(1+u'(z)^2\right)}\,dz+\Theta_i(j),
\end{equation}
and this concludes the proof.
\end{proof}

As a consequence, we prove Proposition \ref{zollbidisk}.

\begin{proof}[Proof of Proposition \ref{zollbidisk}]
Let $S \subset \R^3$ be a Zoll sphere of revolution and $\ell = 2\pi u(z_0)$ be the length of any simple closed geodesic. It follows from Theorem \ref{contactomorphism} and Remark \ref{rmkeq} that the Reeb orbits on $(S^*S,\lambda)$, except from the ones corresponding to the meridians, can be identified with the Reeb orbits on the boundary of a toric domain $\mathbb{X}_\Omega$ equipped with the restriction of the standard Liouville form $\lambda_0$. We recall that $\Omega \subset \R^2_{\geq 0}$ is the region bounded by the coordinate axis and the graph of the function
$$f_S(l) = 2\int_{z_{-}(1,l/2\pi)}^{z_+(1,l/2\pi)} \sqrt{\left(1 - \frac{l^2}{4\pi^2 u(z)^2}\right)(u'(z)^2 + 1)}\, dz,$$
where $z_{\pm}(1,l/2\pi)$ are the solutions of the equation $4\pi^2 u(z)^2-l^2=0$ and $l$ lies in $[0,\ell]$. Since $S$ is Zoll, $f'_S(l)$ must be constant. In fact, if $f'_S$ were not constant, it would imply that $\mathbb{X}_\Omega$ admits Reeb trajectories which are not closed, and consequently, $S$ would do so. Moreover, $f_S(0) = \ell$ and $f_S(\ell)=0$. Hence, $f_S(j) = \ell - l$ for $l \in [0,\ell]$. Now it follows from the proof of Theorem \ref{toricspheres} that $(D^*(S \backslash \{P_N\}),\omega_{can})$ is symplectomorphic to $(\interior \mathbb{X}_{{\Omega}_S},\omega_0)$, where $\Omega_S \subset \R^2_{\geq 0}$ is the region bounded by the coordinate axis and the curve parametrized by
\begin{equation*}
\left(\rho_1(l),\rho_2(l)\right)=    
  \left\{
  \begin{aligned}
        (f_S(l) ,f_S(l)+ l)&, \textup{ if } 0\leq l\leq \ell,\\
        (f_S(-l)- l,f_S(-l))&, \textup{ if } -\ell \leq l\leq0.
   \end{aligned}
    \right.
\end{equation*}
In our present case, it is the same as
\begin{equation*}
\left(\rho_1(l),\rho_2(l)\right)=    
    \left\{
    \begin{aligned}
        (\ell - l ,\ell)&, \textup{ if } 0\leq l\leq \ell,\\
        (\ell ,\ell + l)&, \textup{ if } -\ell \leq l\leq0,
    \end{aligned}
    \right.
\end{equation*}
which parametrizes the rectangle $[0,\ell] \times [0,\ell]$.
\end{proof}

\subsection{The toric domain for ellipsoids of revolution}\label{toricellip}

The proof of Proposition \ref{thm:toric} will be done in two steps: in Proposition \ref{tdomainell} we compute the toric domain provided by Theorem \ref{toricspheres} and in Proposition \ref{ganonincreasing} we analyze its concavity. However, before we head to the proof of Proposition \ref{tdomainell}, we recast the elliptic integral of third kind in a different way as follows, since it will be used in the proof of Proposition \ref{tdomainell}. For the calculation below, we assume that $k<n$.
\begin{equation*}
\begin{split}
\Pi(n,k)&=\int_0^{\pi/2}\frac{d\theta}{(1-n\sin ^2\theta)\sqrt{1-k \sin^2\theta}}\\
&=\int_0^{\pi/2}\sum_{i,j=0}^{\infty}n^i\sin ^{2i}\theta\frac{1}{4^j}\begin{pmatrix}
2j\\
j
\end{pmatrix}
k^j\sin ^{2j}\theta d\theta
\end{split}
\end{equation*}
where we substituted the factors of the integral by the geometric series and its Taylor series expansion, respectively. Now we make the index reparametrization $i+j=\ell$ and some straightforward manipulations and integration termwise to get 
\begin{eqnarray*}
\Pi(n,k)&=& \int_0^{\pi/2}\sum_{\ell=0}^{\infty}\sum_{j=0}^{\ell} \nonumber \frac{n^{\ell-j}}{4^j}\begin{pmatrix}
2j\\
j
\end{pmatrix}k^j\sin ^{2\ell}\theta d\theta\\
&=&\sum_{\ell=0}^{\infty}\sum_{j=0}^{\ell} \frac{n^{\ell-j}}{4^j}\begin{pmatrix}
2j\\
j
\end{pmatrix}k^j\frac{\pi}{2}\frac{1}{4^{\ell}}\begin{pmatrix}
2\ell\\
\ell
\end{pmatrix} \nonumber \\
&=& \int_0^{\pi/2}\sum_{\ell=0}^{\infty}\sum_{j=0}^{\ell} \left(\frac{n}{4}\right)^{\ell}\begin{pmatrix}
2\ell\\
\ell
\end{pmatrix}\left(\frac{k}{n}\right)^j\sin ^{2j}\theta d\theta \nonumber \\
&=& \int_0^{\pi/2}\sum_{\ell=0}^{\infty} \left(\frac{n}{4}\right)^{\ell}\begin{pmatrix}
2\ell\\
\ell
\end{pmatrix}\frac{1-(\frac{k}{n})^{\ell+1}\sin ^{2(\ell+1)}\theta}{1-\frac{k}{n}\sin ^2\theta} d\theta \nonumber
\end{eqnarray*}
\begin{eqnarray}\label{PiRecast}
&=& \int_0^{\pi/2}\left(\frac{1}{\sqrt{1-n}\left(1-\frac{k}{n}\sin^2\theta\right)}-\frac{\frac{k}{n}\sin^2\theta}{\left(1-\frac{k}{n}\sin^2\theta\right)\sqrt{1-k\sin^2\theta}}\right) d\theta \nonumber \\
&=& \frac{\pi}{2}\frac{1}{\sqrt{1-n}}\frac{1}{\sqrt{1-\frac{k}{n}}}+K(k)-\Pi\left(\frac{k}{n},k\right)
\end{eqnarray}

For $j\in [0,1]$, we define
\begin{equation}\label{param}
g_c(j)=-4j^2cK(k_c(j))+4cE(k_c(j))+\frac{4j^2}{c}\Pi\left(1-\frac{1}{c^2},k_c(j)\right)-2\pi j,
\end{equation}
where $k_c(j):= \left(1-\frac{1}{c^2}\right)(1-j^2)$.
\begin{prop}\label{tdomainell}
Let $P_N=(0,0,c)$ and $c>0$. Then $(D^*(\mathcal{E}(1,1,c)\setminus\{P_N\}), \omega_{can})$ is symplectomorphic to the interior of the toric domain $\mathbb{X}_{\Omega_c}$, where $\Omega_c$ is the relatively open set in $\mathbb{R}^2_{\geq 0}$ bounded by the coordinate axes and the curve parametrized by
\begin{equation}\label{resultoblate}
\left(\rho_1(j),\rho_2(j)\right):=    
    \begin{cases}
    \begin{split}
        (g_c(j),g_c(j)+2\pi j)&, \textup{ if } 0\leq j\leq 1,\\
        (g_c(-j)-2\pi j,g_c(-j))&, \textup{ if } -1\leq j\leq0,
    \end{split}
    \end{cases}
\end{equation}
where $g_c:[0,1]\rightarrow \mathbb{R}$ is the function defined in (\ref{param}).
\end{prop}

\begin{proof}
The ellipsoid of revolution $\mathcal{E}(1,1,c)$ is the surface of revolution parametrized by:
\begin{equation}\label{eq:u}u(z)=\sqrt{1-\frac{z^2}{c^2}},\end{equation}
for $z\in [-c,c]$.
Applying Theorem \ref{toricspheres} to $\mathcal{E}(1,1,c)$, we conclude that $D^*(\mathcal{E}(1,1,c)\setminus P_N)$ is symplectomorphic to a toric domain $\mathbb{X}_{\Omega_c}$ where $\Omega_c\subset\R^2_{\ge 0}$ is bounded by the coordinate axes and the curve parametrized by $(\rho_1(j),\rho_2(j))$. The functions $\rho_i(j)$ are given by \eqref{parametrization}. Using \eqref{eq:u}, we obtain
\begin{equation}\label{eq: p_i}
    \rho_i(j)=2\int_{-c\sqrt{1-j^2}}^{c\sqrt{1-j^2}}\frac{\sqrt{(c^2(1-j^2)-z^2)(c^4+(1-c^2)z^2)}}{c(c^2-z^2)}\,dz+\Theta_i(j).
\end{equation}
We now compute the integral above 
\begin{align} 
I_c(j):&=2\int_{-c\sqrt{1-j^2}}^{c\sqrt{1-j^2}}\frac{\sqrt{(c^2(1-j^2)-z^2)(c^4+(1-c^2)z^2)}}{c(c^2-z^2)}\,dz \nonumber \\
&=4\int_{0}^{c\sqrt{1-j^2}}\frac{\sqrt{(c^2(1-j^2)-z^2)(c^4+(1-c^2)z^2)}}{c(c^2-z^2)}\,dz. \label{gcj}
\end{align} 
Making the substitution $\theta=\arcsin\dfrac{z}{c\sqrt{1-j^2}}$ we see that
\begin{align}
I_c(j)=&4\int_0^{\pi/2}\frac{(1-j^2)\cos^2\theta\sqrt{c^2-(c^2-1)(1-j^2)\sin^2\theta}}{1-(1-j^2)\sin^2\theta}\,d\theta\nonumber\\
=&4\int_0^{\pi/2}\frac{(1-(1-j^2)\sin^2\theta-j^2)\sqrt{c^2-(c^2-1)(1-j^2)\sin^2\theta}}{1-(1-j^2)\sin^2\theta}\,d\theta\nonumber\\
=&4cE(k_c(j))-4j^2\int_0^{\pi/2}\frac{c^2-(c^2-1)(1-j^2)\sin^2\theta}{(1-(1-j^2)\sin^2\theta)\sqrt{c^2-(c^2-1)(1-j^2)\sin^2\theta}}\,d\theta\nonumber\\
=&4cE(k_c(j))-4j^2\int_0^{\pi/2}\frac{(c^2-1-(c^2-1)(1-j^2)\sin^2\theta)+1}{(1-(1-j^2)\sin^2\theta)\sqrt{c^2-(c^2-1)(1-j^2)\sin^2\theta}}\,d\theta\nonumber\\
=&4cE(k_c(j))-4j^2\int_0^{\pi/2}\frac{(c^2-1-(c^2-1)(1-j^2)\sin^2\theta)}{(1-(1-j^2)\sin^2\theta)\sqrt{c^2-(c^2-1)(1-j^2)\sin^2\theta}}\,d\theta\nonumber\\
&-4j^2\int_0^{\pi/2}\frac{1}{(1-(1-j^2)\sin^2\theta)\sqrt{c^2-(c^2-1)(1-j^2)\sin^2\theta}}\,d\theta\nonumber\\
=&4cE(k_c(j))-4j^2\frac{c^2-1}{c}\int_0^{\pi/2}\frac{1}{\sqrt{1-(1-1/c^2)(1-j^2)\sin^2\theta}}\,d\theta\nonumber\\
&-\frac{4j^2}{c}\int_0^{\pi/2}\frac{1}{(1-(1-j^2)\sin^2\theta)\sqrt{1-(1-1/c^2)(1-j^2)\sin^2\theta}}\,d\theta\nonumber\\
=&4cE(k_c(j))-\frac{4j^2(c^2-1)}{c}K(k_c(j))-\frac{4j^2}{c}\Pi(1-j^2,k_c(j)),\label{eq:int3}
\end{align}
where $k_c(j):= \left(1-\frac{1}{c^2}\right)(1-j^2)$.\\

Since $1-\frac{1}{c^2}<1$ we can combine (\ref{PiRecast}) and (\ref{eq:int3}) to obtain
\begin{equation}\label{g_cjrecast}
I_c(j)=-4j^2cK(k_c(j))+4cE(k_c(j))+\frac{4j^2}{c}\Pi\left(1-\frac{1}{c^2},k_c(j)\right)-2\pi j.
\end{equation}
By the definition of $g_c(j)$ in (\ref{param}) and the equation above we can see that $I_c(j)=g_c(j)$. From (\ref{integrals}) and (\ref{parametrization}) we get (\ref{resultoblate}). Therefore, $D^*(\mathcal{E}(1,1,c)\setminus\{P_N\})$ is symplectomorphic to  $\mathbb{X}_{\Omega_c}$.
\end{proof}

\begin{prop}\label{ganonincreasing}
The toric domain $\mathbb{X}_{\Omega_c}$ is:
\begin{enumerate}[label=(\roman*)]
\item neither concave nor weakly convex, if $0<c<1$,\label{cga}
\item weakly convex, if $c\geq 1$. \label{aga}
\end{enumerate}
\begin{proof}

From (\ref{param}) we can compute the first and second derivatives of $g_c(j)$ to see that
\begin{equation}\label{derivatives}
\begin{split}
g'_c(j)&=-4cjK(k_c(j))+\frac{4j}{c}\Pi\left(1-\frac{1}{c^2},k_c(j)\right)-2\pi,\\
g''_c(j)&=-\frac{4c}{1-j^2}\left(K(k_c(j))-E(k_c(j))\right).
\end{split}
\end{equation}
For $\rho(j)=(\rho_1(j),\rho_2(j))$ as in (\ref{resultoblate}), and $j\in[0,1]$ we have that the signed curvature is given by
\begin{equation}\label{convtest}
\kappa\left(\rho(j)\right)=-\frac{2\pi g''_c(j)}{((g'_c(j))^2+(2\pi+g'_c(j))^2)^{3/2}},
\end{equation}
and therefore, to analyze convexity/concavity, we just need to study the sign of $g''_c(j)$. \\
For $0<c<1$, notice that $k_c(j)<0$. We can see from \eqref{defellipint} and \eqref{derivatives} that in this case $g''_c(j)\geq 0$. Then (\ref{convtest}) gives us that $\kappa\left(\rho(j)\right)\leq 0$, i.e., both branches of $\rho(j)$ are concave curves. By symmetry with respect to the line $y=x$ we get that $\mathbb{X}_{\Omega_c}$ is neither a weakly convex nor a concave toric domain for $0<c<1$ as claimed in \ref{cga}.\\
For $c=1$, it follows from \cite{ferreira2021symplectic} (and also Proposition \ref{zollbidisk}) that the toric domain $\mathbb{X}_{\Omega}$ is a polydisk, hence a weakly convex domain.\\
For $c>1$, we can see from (\ref{defellipint}) that $K(k_c(j))> E(k_c(j))$ and so (\ref{derivatives}) gives us that $g''_c(j)\leq 0$. Then \eqref{convtest} implies that $\kappa\left(\rho(j)\right)\geq 0$, i.e., both branches of $\rho(j)$ are convex curves which coincide after the reflection about the line $y=x$. Now, to see that $\Omega_c$ is convex we just need to show that 
\begin{equation}\label{eq:ineq}
\lim_{j\to 0^+}\frac{\rho'_2(j)}{\rho'_1(j)}\ge -1.
\end{equation} For $j>0$, it follows from \eqref{resultoblate} that $\rho_1'(j)=g_c'(j)$ and $\rho_2'(j)=g_c'(j)+2\pi$. From  \eqref{derivatives} we obtain $\lim_{j\to 0}g_c'(j)=g_c'(0)=-2\pi$. So
\[\lim_{j\to 0^+}\frac{\rho'_2(j)}{\rho'_1(j)}=\frac{-2\pi+2\pi}{-2\pi}=0,\]
proving \eqref{eq:ineq}. Hence $\Omega_c$ is convex. Therefore $\mathbb{X}_{\Omega_c}$ is weakly convex.
\end{proof}
\end{prop}

\begin{rmk}\label{rmk:verthor}
It follows from the proof of Proposition \ref{ganonincreasing} above that the curve \eqref{resultoblate} is not smooth at $j=0$. In fact, \[\lim_{j\to 0^+}\frac{\rho'_2(j)}{\rho'_1(j)}=0\text{ and }\lim_{j\to 0^-}\bigg|\frac{\rho'_2(j)}{\rho'_1(j)}\bigg|=\infty.\]
\end{rmk}

Before finishing this section we will prove the statement in Proposition \ref{prop:ballemb} for $c\leq 1$.

\begin{lemma}\label{mergulho}
For every $c<1$, there exists a symplectic embedding \[\left(\interior B^4(w(c)),\omega_0\right)\hookrightarrow \left(\interior \mathbb{X}_{\Omega_c},\omega_0\right).\]
\end{lemma}
\begin{proof}
Assume first $c\leq \frac{1}{2}$. We claim that the triangle 
$$\bigtriangleup (\alpha(c)) = \{(x,y) \in \R_{\geq 0}^2 \mid x+y < \alpha(c)\},$$
is contained in $\Omega_c$. To see this notice that the coordinates of the curve $(\rho_1(j),\rho_2(j))$ add up to $2g_c(j)+2\pi j$ for $j\in[0,1]$, by symmetry we just need to consider this branch. Now, the minimum value of this sum is attained at some $j$ such that 
\begin{equation} 
\label{eq:gcj}
g'_c(j)=-\pi.\end{equation}
It is easy to see using (\ref{derivatives}) and (\ref{PiRecast}) that this equation can be recast as equation (\ref{eq:j0c}). A simple calculation using \eqref{derivatives} yields $g_c'(0)=-2\pi$ and $g_c'(1)=-2\pi c$. Moreover, $g''_c(j)>0$ for  $c< 1$, by (\ref{derivatives}), hence $g'_c(j)$ is increasing. So for $c\le \frac{1}{2}$ there exists a unique of $j_0(c)$ satisfying \eqref{eq:gcj} and equivalently \eqref{eq:j0c}. Combining (\ref{g_cjrecast}) and (\ref{eq:j0c}) we obtain $2g_c(j_0(c))+2\pi j_0(c)=\alpha(c)$ and the claim follows. Notice that in this case, the triangle $\bigtriangleup(\alpha(c))$ is tangent to the curve parametrized by $(\rho_1(j),\rho_2(j))$ exactly at the points $(\rho_1(j_0(c)),\rho_2(j_0(c)))$ and $(\rho_1(-j_0(c)),\rho_2(-j_0(c)))$.

Now suppose that $\frac{1}{2}\leq c<1$. Since $g_c'(1)=-2\pi c\le -\pi$ and $g_c''(j)>0$, it follows that $g_c'(j)\le -\pi$ for all $j\in[0,1]$. From \eqref{param} we obtain $g_c(1)=0$. So for $j\in[0,1]$,
\[\rho_1(j)+\rho_2(j)=2g_c(j)+2\pi j=2g_c(1)+2\pi-\int_j^1 (2g_c'(s)+2\pi)ds\ge 2\pi.\]
By symmetry, it follows that the curve \eqref{parametrization} lies to the right of the curve $\rho_1+\rho_2=2\pi$.
%
%\begin{equation}\label{paramderest1}
%\frac{\rho_2'(j)}{\rho_1'(j)} \geq \frac{\rho'_2(1)}{\rho'_1(1)}=\frac{c-1}{c}\geq -1,
%\end{equation}
%for $j>0$. By symmetry,
%\begin{equation}\label{paramderest2}
%\frac{\rho'_2(j)}{\rho'_1(j)} \leq -1,
%\end{equation}
%for $j<0$. We also have that 
%%$$(\rho_1(0),\rho_2(0))=(g_c(0),g_c(0))=(4cE(k_c(0)),4cE(k_c(0))),$$
%\begin{equation}\label{endpoints}
%(\rho_1(1),\rho_2(1))=(0, 2\pi), \quad (\rho_1(-1),\rho_2(-1))=(2\pi,0).
%\end{equation}
%Straightforward manipulation of the elliptic integral of second kind shows that 
%\begin{equation*}
%4cE(k_c(0))=4E(1-c^2)=\beta(c)\
%\end{equation*}
%and using the fact that the function $\beta(c)$ is increasing we get
%\begin{equation}\label{estimative}
%\beta(c)=4E(1)\geq 4.
%\end{equation} 
So the triangle 
$$\bigtriangleup (2\pi) = \{(x,y) \in \R_{\geq 0}^2 \mid x+y < 2\pi\},$$
is contained in $\Omega_c$. Similarly to the previous case, we observe that, in this case, the triangle $\bigtriangleup(\alpha(c))$ is tangent to the curve parametrized by $(\rho_1(j),\rho_2(j))$ exactly on the points $(\rho_1(1)),\rho_2(1))$ and $(\rho_1(-1),\rho_2(-1))$, on the $y$-axis and $x$-axis, respectively.

In both cases, the embedding of the desired ball into the toric domain $\mathbb{X}_{\Omega_c}$ follows immediately.
\end{proof}

\section{Applications of ECH tools}\label{ECHcapacites}
\subsection{ECH capacities of disk cotangent bundle of spheres}
The obstruction confirming that the symplectic embeddings we found are optimal follows from ECH capacities. In this section, we briefly explain the definition of ECH capacities for the special case of the disk cotangent bundle of spheres. For a full definition in the general case of four dimensional symplectic manifolds, see \cite{hutchings2014lecture}.

Let $S \subset \R^3$ be a sphere and denote by
$$S^*S = \{(q,p) \in T^*S \mid \Vert p \Vert =1\}$$
the unit cotangent bundle associated to the Riemannian metric induced from $\R^3$. The restriction of the tautological $1$-form defines a contact form $\lambda$ on $S^*S \cong \R P^3$ such that the Reeb flow agrees with the cogeodesic flow. The contact structure $\xi = \ker \lambda$ is tight and the vector bundle $\xi\to S^*S$ is trivial\footnote{In fact, $(S^*S,\xi)$ is contactomorphic to $(\R P^3,\xi_0)$, where $\xi_0$ is the contact structure induced by the tight $(S^3,\xi_0)$.}. For any nondegenerate contact form $\lambda$ inducing the tight contact structure $\xi$, we can define the embedded contact homology $ECH(S^*S,\xi)$ as the homology of a chain complex $ECC(S^*S,\lambda,J)$ which is defined over $\Z_2$ or $\Z$. The generators of this chain complex are orbit sets, i.e., finite sets $\{(\mathfrak{o}_i,m_i)\}$ where $\mathfrak{o}_i$ are distinct embedded Reeb orbits and $m_i$ are nonnegative integers such that $m_i = 1$ whenever $\mathfrak{o}_i$ is hyperbolic. We often write an orbit set in the product notation: $\Pi_{i=1}^n \mathfrak{o}_i^{m_i}:= \{(\mathfrak{o}_i,m_i)\}$. The differential counts certain $J$-holomorphic curves in the symplectization $S^*S \times \R$ for a generic symplectization-admissible almost complex structure $J$. It turns out that the ECH decomposes into the homology classes in $H_1(S^*S;\Z) \cong \Z_2$, $ECH(S^*S,\lambda) = \oplus_{\Gamma \in H_1(S^*S;\Z)} ECH(S^*S,\lambda,\Gamma)$. Here the singular homology class of a generator $\mathfrak{o}$ is defined as the total homology class
\begin{equation}\label{homologyclass}
[\mathfrak{o}] := \sum_i m_i [\mathfrak{o}_i] \in H_1(S^*S;\Z).
\end{equation}
Moreover, Taubes proved in \cite{taubes2010embedded} that the embedded contact homology does not depend on $\lambda$ nor on $J$, and is, in fact, isomorphic to Seiberg-Witten Floer cohomology. In this case, we have
\begin{align*}
ECH_*(S^*S,\xi,\Gamma) = \begin{cases}
\Z, & \text{if} \ * \in 2\Z_{\geq 0} \\
0, & \text{else}
\end{cases}
\end{align*}
for each $\Gamma \in H_1(S^*S;\Z)$, see e.g. (\cite{ferreira2021symplectic},\S 3.5). Here the grading $*$ is given by the ECH index, we shall recall its definition now. Let $\mathfrak{o} = \{(\mathfrak{o}_i,m_i)\}_{i=1}^n$ be a chain complex generator which is nullhomologous, i.e., $[\mathfrak{o}] = 0$ in \eqref{homologyclass}. There is an absolute $\Z$-grading defined by $\vert \mathfrak{o} \vert = I(\mathfrak{o},\emptyset)$, where $I$ denotes the ECH index. In our specific case, it is given by
\begin{equation}\label{gradingech}
\vert \mathfrak{o} \vert = \sum_{i=1}^n \frac{m_i^2}{4} sl(\mathfrak{o}_i^2) + \sum_{i\neq j} \frac{1}{4}m_i m_j lk(\mathfrak{o}_i^2,\mathfrak{o}_j^2) + \sum_{i=1}^n \sum_{k=1}^{m_i} CZ_\tau(\mathfrak{o}_i^k),
\end{equation}
where $sl(\cdot)$ denotes the transverse self-linking number, $lk(\cdot,\cdot)$ denotes the usual linking number and $CZ_\tau$ is the Conley-Zehnder index with respect to a global trivialization $\tau$ of the contact structure $\xi$. We define the symplectic action of the generator $\mathfrak{o}$ by
$$\mathcal{A}(\mathfrak{o}) = \sum_{i=1}^n m_i \int_{\mathfrak{o}_i} \lambda.$$
Denoting by $\zeta_k$ the generator of $ECH_{2k}(S^*S,\xi,0) \cong \Z$, we define, for each nonnegative integer $k$, the number $c_k(S^*S,\lambda)$ as being the smallest $L \in \R$ such that $\zeta_k$ can be represented in $ECC_*(Y,\lambda,J)$ as a sum of chain complex generators each one with symplectic action less than or equal to $L$.

The definition of the $c_k(S^*S,\lambda)$ can be extended to the degenerate case as follows. When the restriction of the tautological form to $S^*S \cong \R P^3$ is degenerate, one defines
$$c_k(S^*S,\lambda) = \lim_{n \to \infty} c_k(S^*S,f_n \lambda),$$
where ${f_n}_{\geq 1}$ is a sequence of positive functions on $S^*S$ converging to $1$ in the $C^0$ topology and such that $f_n \lambda$ is nondegenerate for all $n$. Finally, the disk cotangent bundle $D^*S$ is a compact manifold with boundary $S^*S$ admitting the symplectic form $\omega_{can} = d\lambda$, where $\lambda$ is the tautological form which restricts to a contact form on $S^*S$. Then, we define the ECH capacities of $D^*S$ by
$$c_k(D^*S,\omega_{can}) = c_k(S^*S,\lambda).$$

\subsection{ECH capacities of disk cotangent bundle of Zoll spheres of revolution}
In \cite[Theorem 1.3]{ferreira2021symplectic}, the first and second author computed the ECH capacities of the disk cotangent bundle of the $2$-sphere with the round metric, these are given by
\begin{equation}\label{capround}
(c_k(D^*S^2,\omega_{can}))_k  = 2\pi M_2(N(1,1)) = (0,4\pi,4\pi,4\pi,8\pi,8\pi,8\pi,8\pi,8\pi, 12 \pi, \ldots).
\end{equation}
The approach followed in that case can be directly adapted to the case of a Zoll metric. Instead of doing this adaptation, we shall use the following characterization of Zoll contact forms on $\R P^3$ due to Abbondandolo, Bramham, Hryniewicz and Salomão.
\begin{theorem}(\cite[Theorem B.2]{abbondandolo2017systolic})\label{zollcharac}
Let $\lambda$ be a contact form on $\R P^3$ such that all the trajectories of the Reeb flow are periodic and have minimal period $2\pi$. Then
$$\int_{\R P^3} \lambda \wedge d\lambda = \int_{\R P^3} \lambda_0 \wedge d \lambda_0 = 8\pi^2,$$
where $\lambda_0$ is the standard contact form on $\R P^3$ corresponding to the restriction of the tautological form to the unit tangent bundle of the round sphere $S^*S^2$. Moreover, there is a diffeomorphism $\phi \colon \R P^3 \to \R P^3$ such that $\phi^*\lambda = \lambda_0$.
\end{theorem}
In particular, ECH capacities corresponding to the round metric in equation \eqref{capround} together Theorem \ref{zollcharac} lead us to ECH capacities corresponding to any Zoll metric on the sphere.
\begin{prop}\label{echcapzoll}
Let $S$ be a Zoll sphere and $\ell$ be the common length of the simple closed geodesics. Then, the ECH capacities of its disk cotangent bundle are given by
$$(c_k(D^*S,\omega_{can}))_k = \ell M_2(N(1,1)) = (0,2\ell,2\ell,2\ell,4\ell,4\ell,4\ell,4\ell,4\ell,6\ell,\ldots).$$
More clearly, $c_k(D^*S,\omega_{can})$ is the $k$-th term in the sequence consisting of values $\ell(m_1+m_2)$, where $m_1,m_2$ are nonnegative integers, $m_1+m_2$ is even, and ordered in non-decreasing order.
\begin{proof}
It follows from what we have seen that
\begin{equation*}
\begin{split}
c_k(D^*S,\omega_{can}) &= c_k(S^*S,\lambda) = c_k\left(\R P^3,\frac{\ell}{2\pi}\lambda_0\right) = c_k\left(S^*S^2,\frac{\ell}{2\pi}\lambda\right)\\ &= \frac{\ell}{2\pi}c_k(D^*S^2,\omega_{can}).
\end{split}
\end{equation*}
\end{proof}
\end{prop}

Now we are ready to prove Theorem \ref{thm:grzoll}.

\begin{proof}[Proof of Theorem \ref{thm:grzoll}]
Let $\ell$ be the length of the simple closed geodesics on $S$. It follows from Proposition \ref{ball} that
\[c_{Gr}\left(D^*S,\omega_{can}\right)\geq \ell.\]
On other hand, if $(B^4(a),\omega_0)\hookrightarrow \left(D^*S,\omega_{can}\right)$, we have
\[2a=c_3(B^4(a))\le c_3\left(D^*S,\omega_{can}\right)=2 \ell.\]
So taking the supremum on $a$, we conclude that
\[c_{Gr}\left(D^*S,\omega_{can}\right)\leq \ell.\]
\end{proof}

\subsection{ECH capacities for $D^*\mathcal{E}(1,1,c)$}\label{sec:echcap}
In this section we compute the ECH capacities that we use to give upper bounds on $c_{Gr}(D^*\mathcal{E}(1,1,c),\omega_{can})$, proving Proposition \ref{prop:capacities}. Recall that the tautological one-form restricts to a contact form $\lambda$ on $\partial D^*\mathcal{E}(1,1,c) = S^*\mathcal{E}(1,1,c)$ such that the Reeb flow agrees with the (co)geodesic flow. The geodesic flow on an ellipsoid has an interesting behavior and is well known, see e.g.\ \cite{klingenberg1995riemannian}. By Theorem \ref{contactomorphism}, Remark \ref{rmkeq} and Proposition \ref{tdomainell}, we can use the boundary of the toric domain $\mathbb{X}_{\Omega_c}$ to understand this flow. In particular, the Reeb orbits of $(S^*\mathcal{E}(1,1,c),\lambda)$ are the Reeb orbits of $(\partial \mathbb{X}_{\Omega_c}, \lambda_0)$ (including the two corresponding to the equator) together with the orbits corresponding to the meridians. From now on, we shall call the latter orbits by meridians and similarly for the equators. The contact form $\lambda$ is degenerate. The orbits on $\mathbb{X}_{\Omega_c}$ come in two types. To describe these orbits, consider the moment map for the standard torus action in $\mathbb{C}^2$
\begin{eqnarray*}
\mu \colon \C^2 &\to& \R^2_{\geq 0} \\
(z_1,z_2) &\mapsto& (\pi\vert z_1 \vert^2, \pi \vert z_2 \vert^2).
\end{eqnarray*}
It follows from a well-known calculation that the Reeb flow of $\lambda_0$ on a star-shaped toric domain in $\C^2$ preserves $\mu^{-1}(x,y)$. In fact, the Reeb vector field is parallel to 
$$R_z = \nu_1(x,y)\frac{\partial}{\partial \theta_1}+\nu_2(x,y)\frac{\partial}{\partial \theta_2},$$
for each $z \in \mu^{-1}(x,y)$. Here $(x,y) \in \partial \Omega_c \cap \R^2_{>0}$ and $(\nu_1(x,y),\nu_2(x,y))$ denotes a normal vector to $\partial \Omega_c$ at $(x,y)$. Then $\mathbb{X}_{\Omega_c} = \mu^{-1}(\Omega_c)$ and the Reeb orbits are given by
\begin{itemize}
\item The two circles $\mu^{-1}(2\pi,0)$ and $\mu^{-1}(0,2\pi)$, corresponding to the intersections of the branches $j<0$ and $j>0$ in \eqref{resultoblate} with the $x$-axis and $y$-axis, respectively. These are elliptic (degenerate for $c=1$) and correspond to the two equators. The actions of both coincide with the length of the equator on the ellipsoid and is, therefore, given by $2\pi$.
\item For each $(x,y) \in \partial \Omega_c$ for which the tangent line to $\Omega_c$ through this point has rational slope, the torus $\mu^{-1}(x,y)$ is foliated by an $S^1$-family of Reeb orbits. Each simple orbit of this family has action given by $\langle (x,y), (p,q) \rangle = px + q y$, where $(p,q) \in \mathbb{Z}^2$ is a rescaling of the outward normal vector to $\Omega_c$ at $(x,y)$, $(\nu_1,\nu_2) \in \mathbb{Q}^2$, such that $p$ and $q$ are relatively prime integers. In this case, these orbits are $(p,q)$-torus knots on the torus $\mu^{-1}(x,y)$ and have rotation number $p+q$, i.e., the dynamical rotation number of the linearized Reeb flow around each of these orbits is given by $p+q$. Comparing the normal vectors to \eqref{curvec} and \eqref{resultoblate}, it follows that the Reeb vector field is given by
\[R=\left\{\begin{aligned} q\frac{\partial}{\partial \phi_1}+(p+q)\frac{\partial}{\partial \phi_2}, \text{ if }j>0,\\
p\frac{\partial}{\partial \phi_1}+(p+q)\frac{\partial}{\partial \phi_2}, \text{ if }j<0.\end{aligned}\right.\]
Here $(\phi_1,\phi_2)$ are the angle coordinates obtained in the proof of Theorem \ref{contactomorphism}. Therefore the Reeb orbits in $\mu^{-1}(x,y)$ project to geodesics on the ellipsoid that intersect the equator at $p+q$ pairs of points and wind around the $z$-axis $q$ times or $p$ times if $j>0$ or $j<0$, respectively. In particular, the singular homology class of such an orbit is given by $q \textup{ mod }2$ or $p \textup{ mod }2$ in $H_1(S^*\mathcal{E}(1,1,c)) \cong \Z_2$, respectively.
\end{itemize}

Moreover, the remaining Reeb orbits on $S^*\mathcal{E}(1,1,c)$, i.e., the meridians, form another $S^1$-family of embedded Reeb orbits. In the toric domain this $S^1$-family is represented by the intersection of $\partial \Omega_c$ with the line $y=x$. Fix $T>0$, then one can perturb $\lambda$ to obtain a contact form $\lambda_\varepsilon$ defining the same contact structure and such that every Reeb orbit $\gamma$ with action $\mathcal{A}(\gamma) = \int_\gamma \lambda_\varepsilon \leq T$ is nondegenerate. In fact, away from the meridians,  we can replace $\lambda$ with $g \lambda$, where $g\colon S^*\mathcal{E}(1,1,c) \to \R_{>0}$ is a smooth function $\varepsilon$-close to $1$ in $C^0$ topology, satisfying the following property. Each circle of Reeb orbits coming from $\partial \mathbb{X}_{\Omega_c}$ with action $< 1/\varepsilon$ becomes two Reeb orbits (one elliptic and one hyperbolic) of approximately the same action, no other Reeb orbits of action $< 1/\varepsilon$ are created and the circles $\mu^{-1}(2\pi,0)$ and $\mu^{-1}(0,2\pi)$, are unchanged. This can be done following the approach in \cite[$\S 3.3$]{choi2014symplectic}. Moreover, near the meridians, we can approximate $\lambda$ by the contact form corresponding to the unit cotangent bundle of the triaxial ellipsoid $S^*\mathcal{E}(1-\varepsilon,1+\varepsilon,c)$. In this case, the $S^1$-family of meridians turns into four Reeb orbits corresponding to the meridians on $\mathcal{E}(1-\varepsilon,1+\varepsilon,c)$ in both directions.  It follows from the definition of the ECH spectrum $c_k(S^*S,\lambda)$ in the degenerate case that we can use this perturbation of $\lambda$ to compute $c_k(S^*\mathcal{E}(1,1,c),\lambda)$.

We now prove Proposition \ref{prop:capacities}.
\begin{proof}[Proof of Proposition \ref{prop:capacities}]
In this proof, we denote the two orbits corresponding to a distinguished geodesic in both directions by $\gamma$ and $\overline{\gamma}$. In particular, we denote the two equators by $\gamma_1$ and $\overline{\gamma}_1$. The Conley-Zehnder index $CZ_\tau(\gamma)$ is equal to $CZ_\tau(\overline{\gamma})$ and it agrees with the Morse index of the corresponding geodesic on the ellipsoid, see e.g. \cite[Proposition 1.7.3]{eliashberg2000introduction}. Now note that the embedding given in Proposition \ref{prop:ballemb} and the monotonicity of ECH capacities yield the bound
\begin{equation}\label{w(c)}
2 w(c) = c_3(\interior B^4(w(c)),\omega_0) \leq c_3(D^*\mathcal{E}(1,1,c),\omega_{can}).
\end{equation} 
Further, we note that it is simple to find an explicit symplectic embedding
$$(D^*\mathcal{E}(1,1,c),\omega_{can}) \hookrightarrow (D^*S^2,\omega_{can}),$$
whenever $c \leq 1$, by simply lifting the diffeomorphism\footnote{the diffeomorphism given by radial projection, regarding $S^2$ and $\mathcal{E}(1,1,c)$ as naturally included in $\R^3$.} from $S^2$ to $\mathcal{E}(1,1,c)$ to a symplectic map between their cotangent bundles. Hence, monotonicity and Proposition \ref{echcapzoll} yield the upper bound
\begin{equation}\label{leq4pi}
c_3(D^*\mathcal{E}(1,1,c),\omega_{can}) \leq 4\pi
\end{equation}
for any $c \leq 1$.
\begin{enumerate}[label=(\alph*)]
\item 
For $1/2 \leq c \leq 1$, \eqref{w(c)} and \eqref{leq4pi} give us $c_3(D^*\mathcal{E}(1,1,c),\omega_{can})=4\pi$.

Let $0<c<1/2$. From \eqref{w(c)} and the definition of $w(c)$ in \eqref{eq:w}, we get $2\alpha(c) \leq c_3(D^*\mathcal{E}(1,1,c),\omega_{can})$. Together with \eqref{leq4pi}, we obtain
$$2\alpha(c) \leq  c_3(D^*\mathcal{E}(1,1,c),\omega_{can}) \leq 4\pi.$$
Given a rational outward normal vector $(\nu_1,\nu_2)$ to $\partial \Omega_c$, we denote by $(p,q) \in \Z^2$ its integral multiple such that $p$ and $q$ are relatively prime integers. The proof of Proposition \ref{ganonincreasing} shows that the boundary of $\Omega_c \cap \R^2_{>0}$ consists of two concave branches for $c< 1$. In particular, $p$ and $q$ are always positive. Under our perturbation, the $S^1$-family of Reeb orbits foliating the corresponding torus becomes two embedded Reeb orbits of approximately the same action: one elliptic and one hyperbolic, denoted by $e_{p,q}$ and $h_{p,q}$, respectively. For such orbits, one can compute the linking numbers and Conley--Zehnder indices as follows:
\begin{align*}
sl(o_{p,q}) &= -p -q + pq \\
lk(e_{p,q},h_{p,q}) &= pq \\
lk(o_{p,q},o_{p^\prime,q^\prime})&= \min \{pq^\prime,p^\prime q\} \\
CZ_\tau(h_{p,q}) &= CZ_\tau(\overline{h}_{p,q})= 2(p+q) \\
CZ_\tau(e_{p,q}) &= CZ_\tau(\overline{e}_{p,q})= 2 \lfloor (p+q) - \rho(\varepsilon) \rfloor + 1 = 2(p+q)-1,
\end{align*}
where $o_{p,q}$ denotes any element in $\{e_{p,q},\overline{e}_{p,q},h_{p,q},\overline{h}_{p,q}\}$ and $\rho(\varepsilon)>0$ denotes the small change in the rotation number that appears after the perturbation on the contact form. We note that $\rho(\varepsilon)$ is positive since for $c<1$ the boundary of $\Omega_c \cap \R^2_{>0}$ consists of two concave arcs. For more details around this discussion and these computations, see \cite{choi2014symplectic}. Here we are using that the Conley--Zehnder index of a Reeb orbit $\gamma$ with rotation number $\theta$ is given by $CZ_\tau(\gamma) = 2\theta$ if $\gamma$ is hyperbolic and $CZ_\tau(\gamma) = 2\lfloor \theta \rfloor +1$ if $\gamma$ is elliptic.

We note that the self-linking number $sl(o_{p,q}) = -p-q+pq$ is nonnegative provided $\min\{p,q\}>1$. Therefore, the sum of the linking numbers terms in the grading $\eqref{gradingech}$ is nonnegative and the degree of an orbit set is bounded from below by the sum of the Conley--Zehnder indices of the orbits that constitute the orbit set.

The proof of Lemma \ref{mergulho} makes clear that $\alpha(c)$ is the action of any Reeb orbit in the torus corresponding to a point in $\Omega_c$ which the tangent line passing through it has rational slope $-1$. There are two such points. After our perturbation these two circles of orbits become the four distinguished orbits $e_{1,1}, \overline{e}_{1,1}, h_{1,1}$ and $\overline{h}_{1,1}$. It follows from the definition that $c_3(S^*\mathcal{E}(1,1,c),\lambda_\varepsilon)$ is the action of a nullhomologous orbit set with degree $6$. A simple computation yields
\begin{eqnarray*}
\vert e_{1,1} \overline{e}_{1,1} \vert &=& \frac{1}{4}\left(sl(e_{1,1}^2) + 2 lk(e_{1,1}^2,\overline{e}_{1,1}^2) + sl(\overline{e}_{1,1}^2)\right) + CZ_\tau(e_{1,1}) + CZ_\tau(\overline{e}_{1,1}) \\ &=& \frac{1}{4}(-2 + 4 -2) + 3 + 3 = 6.
\end{eqnarray*}
Moreover, from the definition of $c_k$ for the degenerate case, $c_3(S^*\mathcal{E}(1,1,c),\lambda_\varepsilon)$ must be sufficiently close to $c_3(S^*\mathcal{E}(1,1,c),\lambda) \in [2\alpha(c),4\pi]$ provided $\varepsilon$ is sufficiently close to $0$. Together with the discussion above, one concludes that the candidate orbit sets to realize $c_3(S^*\mathcal{E}(1,1,c),\lambda_\varepsilon)$ are:
\begin{itemize}
\item The orbit set $\gamma_1 \overline{\gamma}_1$ with action (with respect to $\lambda_\varepsilon$) $2\pi+2\pi = 4\pi$;
\item The orbit set $e_{1,1} \overline{e}_{1,1}$ with action (with respect to $\lambda_\varepsilon$) sufficiently close to $2\alpha(c)$;
\item An orbit set consisting in one equator and one meridian, with action (with respect to $\lambda_\varepsilon$) sufficiently close to $2\pi+\beta(c)$.
\end{itemize}
In fact, the other Reeb orbits arising from the points with rational slope in the boundary of $\Omega_c \cap \R^2_{>0}$ have a large rotation number, yielding a large Conley-Zehnder index. In particular, whenever $c<1/m$ with $m \in \mathbb{N}$, there exists a Reeb orbit on the torus associated with the rational slope $-1/(m-1)$, corresponding to a closed simple geodesic on the ellipsoid that intersects the equator $2m$ times. The number $\alpha(c)$ is given by the action (with respect to $\lambda$) of such an orbit for $m=2$, cf. Remark \ref{rmk:length}.

We note that a nullhomologous orbit set consist of more than one Reeb orbit or of a Reeb orbit with large rotation number (namely, $q$ is even and thus $\geq 2$ for $j>0$). Hence, the other nullhomologous orbit sets either have total action outside the interval $[2\alpha(c),4\pi]$, or they have a large ECH index. 

We now discard the first and third bullet as candidates. Since
$$\frac{\rho'_2(1)}{\rho'_1(1)}= \frac{g_c'(1)+2\pi}{g_c'(1)}=\frac{-2\pi(c-1)}{-2\pi c}=\frac{c-1}{c},$$ an equator has rotation number given by $1 - (c-1)/c = 1/c$. Hence, one computes
$$CZ_\tau(\gamma_1) = CZ_\tau(\overline{\gamma}_1) = 2 \left\lfloor \frac{1}{c} \right\rfloor + 1.$$
In particular, for $0<c<1/2$, we obtain $CZ_\tau(\gamma_1) = CZ_\tau(\overline{\gamma}_1) \geq 5$. Thus,
$$\vert \gamma_1 \overline{\gamma}_1\vert = \frac{1}{4}\left(sl({\gamma_1}^2) + 2 lk({\gamma_1}^2,{\overline{\gamma}_1}^2) + sl({\overline{\gamma}_1}^2)\right) + CZ_\tau(\gamma_1) + CZ_\tau(\overline{\gamma}_1) \geq 5+5 = 10 > 6.$$
Here we are using that the two equators on $S^*\mathcal{E}(1,1,c)\cong \R P^3$ form a Hopf link when lifted to $S^3$. Therefore, the orbit set $\gamma_1\overline{\gamma}_1$ cannot realize $c_3(S^*\mathcal{E}(1,1,c),\lambda_\varepsilon)$. 

We claim that $2\alpha(c) > 2\pi + \beta(c)$ for $0<c<1/2$. To see that, we first note that $\alpha(c)$ and $\beta(c)+2\pi$ are increasing functions of $c$. In fact, $\alpha(c)=2g_c(j_0(c))+2\pi j_0(c)$, where $j_0(c)$ is the value for which $g_c'(j_0(c))=-\pi$. So \begin{equation}\label{eq:inc}\alpha'(c)=2\frac{\partial g_c}{\partial c}(j_0(c))+2g_c'(j_0(c))j_0'(c)+2\pi j_0'(c)=2\frac{\partial g_c}{\partial c}(j_0(c)).\end{equation}
From \eqref{eq:int3} we know that
\[g_c(j)=I_c(j)=4\int_0^{\pi/2} \frac{(1-j^2)\cos^2\theta\sqrt{c^2(1-(1-j^2)\sin^2\theta)+(1-j^2)\sin^2\theta}}{1-(1-j^2)\sin^2\theta}d\theta.\]
So for fixed $j$, $g_c(j)$ is increasing in $c$. Hence \eqref{eq:inc} implies that $\alpha(c)$ is increasing. Moreover,
\[\beta(c)=4E(1-c^2)=4\int_0^{\pi/2}\sqrt{1-(1-c^2)\sin^2\theta\,}d\theta=4\int_0^{\pi/2}\sqrt{\cos^2\theta+c^2\sin^2\theta}\,d\theta.\]
So $\beta(c)$ is increasing.
Now to prove the claim, it suffices to show that
\begin{equation}\label{eq:ineqalphbeta}
2\lim_{c\to 0}\alpha(c)>2\pi+\beta(1/2).
\end{equation}
We first compute $\lim_{c\to 0}g_c'(j)$. Making the substitution $u=z/c$ to \eqref{eq: p_i}, we obtain
\begin{equation}\label{eq:gcj2}
g_c(j)=I_c(j)=4\int_0^{\sqrt{1-j^2}}\frac{\sqrt{(1-j^2-u^2)(c^2+(1-c^2)u^2)}}{1-u^2}du.
\end{equation}
Using \eqref{eq:gcj2}, we can extend $g_c(j)$ to $c=0$. Taking the derivative with respect to $j$, we obtain
\begin{equation*}
g_c'(j)=-4\int_0^{\sqrt{1-j^2}}\frac{j\sqrt{c^2+(1-c^2)u^2}}{(1-u^2)\sqrt{1-j^2-u^2}}du.
\end{equation*}
So
\begin{equation*}
\begin{aligned}
g_0'(j)&=-4\int_0^{\sqrt{1-j^2}}\frac{ju}{(1-u^2)\sqrt{1-j^2-u^2}}du=4\arctan\left(\frac{\sqrt{1-j^2-u^2}}{j}\right)\Bigg|_0^{\sqrt{1-j^2}}\\&=-4\arctan \left(\frac{\sqrt{1-j^2}}{j}\right).
\end{aligned}
\end{equation*}
It follows that $g_0'(j)=-\pi$ if, and only if, $j=\frac{\sqrt{2}}{2}$. So $j_0(0)=\frac{\sqrt{2}}{2}$. Hence
\begin{equation}\label{eq:alpha0}
\begin{aligned}\lim_{c\to 0}\alpha(c)&=2g_0\left(\frac{\sqrt{2}}{2}\right)+2\pi\cdot\frac{\sqrt{2}}{2}=8\int_0^{\frac{\sqrt{2}}{2}}\frac{u\sqrt{\frac{1}{2}-u^2}}{1-u^2}\,du+\sqrt{2}\pi\\
&=\sqrt{2}(4-\pi)+\sqrt{2}\pi=4\sqrt{2}>5.65.\end{aligned}
\end{equation}
Moreover,
\begin{equation}\label{eq:beta12}
\beta(1/2)+2\pi=4E(3/4)+2\pi<4\cdot 1.22+2\cdot 3.15=11.18=2\cdot 5.59.
\end{equation}
Using \eqref{eq:alpha0} and \eqref{eq:beta12}, it follows that $2\alpha(0)>\beta(1/2)+2\pi$ and hence $2\alpha(c)>\beta(c)+2\pi$ for all $c\in(0,1/2)$.

Therefore, $c_3(S^*\mathcal{E}(1,1,c),\lambda_\varepsilon) = \mathcal{A}(e_{1,1}\overline{e}_{1,1})$ for sufficiently small $\varepsilon>0$. It follows that $c_3(S^*\mathcal{E}(1,1,c),\lambda) = 2\alpha(c)$ for all $c\in(0,1/2)$.

Now let $1 <c < \beta^{-1}(4\pi)$. It follows from \eqref{w(c)} that
\begin{equation}\label{lowerbeta}
2\beta(c) = 2 w(c) = c_3(\interior B^4(w(c)),\omega_0) \leq c_3(D^*\mathcal{E}(1,1,c),\omega_{can}).
\end{equation}
Hence, $c_3(D^*\mathcal{E}(1,1,c),\omega_{can}) \geq 2 \beta(c)>4\pi$. It is consequence\footnote{This follows from the definition in the nondegenerate case and from a compactness argument when $\lambda$ is degenerate.} of \cite[Lemma 2.4]{irie2015dense} that $c_3(D^*\mathcal{E}(1,1,c),\omega_{can})$ is given by the total action of an actual Reeb orbit set for $(S^*\mathcal{E}(1,1,c),\lambda)$ whose total homology class is trivial. Investigating all possible orbit sets, we shall conclude that the orbit set must consist in two meridians. Note that $2 \beta(c) < 2 \beta(\beta^{-1}(4\pi)) = 8\pi$ for $c<\beta^{-1}(4\pi)$.

Since the Reeb orbits on $(S^*\mathcal{E}(1,1,c),\lambda)$ can be interpreted as the Reeb orbits of the toric domain $\mathbb{X}_{\Omega_c}$, and we know every orbit in this latter case, we can understand all the orbits with action less than $8 \pi$. In fact, for $j>0$, any embedded Reeb orbit has action given by $pg_c(j)+(g_c(j)+2\pi j)q$, where $(p,q) \in \Z^2$ is the integral multiple of the outward normal vector to $\Omega_c$ at the point $(g_c(j),g_c(j)+2\pi j)$ with $p$ and $q$ relatively prime. 

For $c>1$, it follows from the proof of Proposition \ref{ganonincreasing} that the boundary of $\Omega_c \cap \R^2_{>0}$ consists of two convex arcs with tangential slopes lying in $(0,(c-1)/c)$ for $j>0$. Since $c< \beta^{-1}(4\pi)<3$, we have $(c-1)/c<2/3$. Therefore, $3\lvert p \rvert < 2q$, and in particular $q > \lvert p \rvert$. Fix $p<0$ and $q>\lvert p \rvert > 0$, and consider the auxiliary function
\[
H_c(j) = (q-\lvert p \rvert) g_c(j) + 2\pi j q, \quad \ j\in[0,1].
\]

We have $H_c^\prime(j) = (q-\lvert p \rvert) g_c^\prime(j) + 2\pi q$ and $H_c^{\prime \prime}(j) = (q-\lvert p \rvert) g_c^{\prime \prime}(j)$. In particular, since
\[
g_c^{\prime \prime}(j) = -\frac{4c}{1-j^2}\left(K(k_c(j))-E(k_c(j))\right) \leq 0
\quad \text{for } c>1,
\]
the function $H_c(j)$ has a unique critical point $j_c \in (0,1)$. This point satisfies
\[
g_c^\prime(j_c) = -\frac{2\pi j_c q}{q-\lvert p \rvert},
\]
or equivalently, it is the point at which the tangent slope of the curve $(g_c(j),g_c(j) + 2\pi j)$ equals $\lvert p \rvert /q$. In this case, $j_c$ is a global maximum for $H_c$. Therefore,
\begin{equation}\label{lowerbound}
H_c(j_c) = (q-\lvert p \rvert) g_c(j_c) + 2\pi j_c q > \max \{H_c(1),H_c(0)\} = \max\{ 2\pi q, (q-\lvert p \rvert) \beta(c)\},
\end{equation}
where we used that $g_c(0) = \beta(c)$.

By the previous discussion, for $p$ and $q$ relatively prime, the number $H_c(j_c)$ is the action of an embedded Reeb orbit lying in the torus corresponding to the rational tangent slope $\lvert p \rvert / q$. For the upper bound on the action, we must have
\begin{equation}\label{actionineq}
2\pi q < H_c(j_c) < 8\pi,
\end{equation}
and hence $q<4$. Since for $c< \beta^{-1}(4\pi)<3$ we also have $3\lvert p \rvert < 2q$, the only possibilities of embedded orbits with action less than $8\pi$ are:

\begin{itemize}
\item The two equators $\gamma_1, \overline{\gamma}_1$ with action $2 \pi$ each;
\item The meridians with action $\beta(c)<4\pi$;
\item Embedded Reeb orbits lying in the torus corresponding to the rational slope $1/2$ with action $a_1 \in (4\pi, 2\beta(c))$;
\item Embedded Reeb orbits lying in the torus corresponding to the rational slope $1/3$ with action $a_2 \in (6\pi, 3\beta(c))$.
\end{itemize}

Here we use that the function $g_c(j) + 2\pi j$ is decreasing in $j$, and thus
\[
H_c(j) < q(g_c(j) + 2\pi j) \leq q g_c(0) = q\beta(c).
\]

We recall that the equators and the meridians belong to the nontrivial homology class in $H_1(S^*\mathcal{E}(1,1,c);\Z)\cong \Z_2$. Also, the orbits with action $a_1$ are null-homologous, while $a_2$ corresponds to homologically nontrivial orbits. Further, we recall that $\beta(c)$ is simply the length of a meridian and that $\beta(1) = 2\pi$. 

Analyzing all the possibilities of null-homologous orbit sets with action less than $8\pi$, one obtains the possibilities $4\pi, 2\pi + \beta(c), 2\beta(c)$, and $a_1$. Combining this with continuity in $c$, the fact that $c_3(D^*\mathcal{E}(1,1,1),\omega_{\mathrm{can}})= 4\pi$, and the lower bound given by $2\beta(c)$ in \eqref{lowerbeta}, one concludes that
\[
c_3(D^*\mathcal{E}(1,1,c),\omega_{\mathrm{can}}) = 2\beta(c)
\]
for $0<c<\beta^{-1}(4\pi)$, i.e., for all $c$ such that $2\beta(c) \leq 8\pi$. This completes the proof that
\[
c_3(D^*\mathcal{E}(1,1,c),\omega_{\mathrm{can}})=2 w(c) \quad \text{for } 0<c<\beta^{-1}(4\pi).
\]

\item Let $c>1$. First, we claim that $\mathfrak{o}= \gamma_1 \overline{\gamma}_1$ is the nullhomologous orbit set with the least possible action. In fact, it is simple to check that $\gamma_1$ (or $\overline{\gamma}_1$) is a Reeb orbit with minimal action; equivalently, the equator is the shortest closed geodesic on $\mathcal{E}(1,1,c)$; see \eqref{actionineq}. Still, it could be the case that an orbit set consisting of a single embedded nullhomologous orbit, say $\widetilde{\gamma}$, with action lower than $\mathcal{A}(\mathfrak{o}) = 4\pi$, exists. As mentioned before, the meridians have nontrivial homology class and hence $\widetilde{\gamma}$ could not be a meridian. In this case, it should be an orbit coming from a $S^1$-family foliating a torus corresponding to a rational slope of a tangent line to the boundary $\partial \Omega_c$. To be nullhomologous, $\widetilde{\gamma}$ must be a $(p,2k)$-torus knot with $p$ and $2k$ relatively prime and thus, by the discussion above, $\mathcal{A}(\widetilde{\gamma}) \geq 2\pi (2k) \geq 4\pi$. Therefore, $\mathfrak{o}$ is a nullhomologous orbit set with the least possible action.

Now, to conclude that $c_1(S^*\mathcal{E}(1,1,c),\lambda_\varepsilon) = \mathcal{A}(\mathfrak{o}) = 4\pi$,  it is enough to prove that the orbit set $\mathfrak{o}= \gamma_1 \overline{\gamma}_1$ represents the generator $\zeta_1$ of $ECH_{2}^{<1/\varepsilon}(S^*\mathcal{E}(1,1,c),\lambda_\varepsilon,0)$. For this, first we check that $\mathfrak{o}$ has the correct degree:
\begin{eqnarray*}
\vert \mathfrak{o} \vert &=& \frac{1}{4}\left(sl({\gamma_1}^2) + 2 lk({\gamma_1}^2,{\overline{\gamma}_1}^2) + sl({\overline{\gamma}_1}^2)\right) + CZ_\tau(\gamma_1) + CZ_\tau(\overline{\gamma}_1) \\ &=& \frac{1}{4}(-2 + 4 -2) + 1 + 1 = 2.
\end{eqnarray*}
Moreover, it follows from the properties that the ECH differential decreases action and decreases the grading by $1$ that $\mathfrak{o}$ is closed, i.e., $\partial \mathfrak{o} = 0$. To see that $\mathfrak{o}$ is not exact, suppose that there is an ECH chain complex generator $\mathfrak{u}$ such that $\langle \partial \mathfrak{u}, \mathfrak{o} \rangle \neq 0$. In this case, there must exist an embedded pseudoholomorphic curve $C$ in the symplectization of $S^*\mathcal{E}(1,1,c)$ with $p_+(C)$ positive punctures converging to $\mathfrak{u}$ and exactly $2$ negative punctures converging to $\mathfrak{o}$ with Fredholm index given by
$$\mathrm{ind}(C) = -\chi(C) + CZ_\tau^{\mathrm{ind}}(\mathfrak{u}) - CZ_\tau(\gamma_1) - CZ_\tau(\overline{\gamma}_1) = 1,$$
where $CZ_\tau^{\mathrm{ind}}(\mathfrak{u}) = \sum_{i=1}^{p_+(C)} CZ_\tau(\mathfrak{u}_i^+)$. Thus,
$$1 = 2g(C) + p_+(C) + CZ_\tau^{\mathrm{ind}}(\mathfrak{u}) - 2,$$
where $g(C)$ denotes the genus of the curve $C$. Since the Conley-Zehnder index of any orbit for our contact form is positive, we must have $g(C) = 0$ and $p_+(C) = 1$. Hence, $\mathfrak{u}$ consists in a single nullhomologous orbit, say $\widetilde{\gamma}$, with $CZ_\tau(\widetilde{\gamma}) = 2$. Further, $sl(\widetilde{\gamma}) = 1$ since
$$ 3 = \vert \mathfrak{u} \vert = sl(\widetilde{\gamma}) + CZ_\tau(\widetilde{\gamma}).$$
Now we recall that we found that the boundary of $\Omega_c \cap \R^2_{>0}$ consists of two convex branches for $c>1$ in the proof of Proposition \ref{ganonincreasing}. In this case, it follows that
$$sl(o_{p,q}) = -\vert p \vert - \vert q \vert + \vert pq \vert = -q + p - pq,$$
where we denote by $o_{p,q}$ any element in $\{e_{p,q},\overline{e}_{p,q},h_{p,q},\overline{h}_{p,q}\}$ again. In particular, for any $c>3$, there are two Reeb orbits with Conley-Zehnder index $2$ and self-linking number $1$, namely the two hyperbolic orbits $h_{-2,3}$ and $\overline{h}_{-2,3}$ arising from the two tori corresponding to the rational slope $2/3$. Nevertheless, these orbits belong to the homology class containing the equator, i.e., the nonzero class in $H_1(S^*\mathcal{E}(1,1,c);\Z) \cong \Z_2$. In particular, a generator $\mathfrak{u}$ as above cannot exist and $\mathfrak{o}$ is not exact.
\end{enumerate}
\end{proof}
\subsection{Symplectic embeddings and balls packings}
We start this section recalling a version of Cristofaro-Gardiner's result about the existence of symplectic embeddings from concave into (weakly) convex toric domains.

Let $\mathbb{X}_\Omega$ be a weakly convex toric domain. We construct a weight sequence as follows. For $b>0$, let $T(b)\subset \R^2$ denote the triangle with vertices $(0,0)$, $(b,0)$ and $(0,b)$. In particular $\mathbb{X}_{T(b)}=B^4(b)$. Let $w_0>0$ be the infimum of $w$ such that $\Omega\subset T(w)$. Let $\Omega_1$ and $\Omega_2$ be the closures of the components of $T(w_0)\setminus \Omega$ containing\footnote{It is possible that $\Omega_1=\emptyset$ or $\Omega_2=\emptyset$.} $(w_0,0)$ and $(0,w_0)$, respectively. Then $\Omega_1$ and $\Omega_2$ are affinely equivalent to the moment map image of concave toric domains. In fact, after translating $\Omega_1$ by $(-b,0)$ and multiplying it by $\begin{pmatrix}0&1\\-1&-1\end{pmatrix}$ we obtain a region $\Omega_1'\subset \R^2$ such that $\mathbb{X}_{\Omega_1}$ is a concave toric domain. We can proceed analogously with $\Omega_2$ to obtain $\Omega_2'$. Now we let $w_1$ be the supremum of $w$ such that $T(w)\subset \Omega_1'$. We define $w_2$ analogously for $\Omega_2'$. The sets $\Omega_1'\setminus T(w_1)$ and $\Omega_2'\setminus T(w_2)$ each have at most two connected components, all of whose closures are affinely equivalent to moment map images of concave toric domains. Inductively, we obtain a sequence of numbers
\[(w_0;w_1,w_2,\dots),\]
which is called the weight sequence of $\mathbb{X}_\Omega$.
From the weight sequence, we can construct a symplectic embedding
\begin{equation}
\label{eq:bp}X_\Omega\sqcup\bigsqcup_{i=1}^{\infty} B^4(w_i)\hookrightarrow B^4((1+\epsilon)w_0),\end{equation} for every $\epsilon>0$ as follows. Let $\epsilon>0$. We first recall that in  \cite{traynor1995symplectic} Traynor constructed an explicit symplectic embedding $B^4(c)\to \mathbb{X}_{T_\epsilon(c)}$, where $T_\epsilon(c)$ is a triangle with vertices $(\epsilon,\epsilon)$, $(c+2\epsilon, \epsilon)$ and $(\epsilon,c+2\epsilon)$. It is a standard calculation to verify that $\mathbb{X}_{T_\epsilon(c)}$ is symplectomorphic to $\mathbb{X}_{F(T_\epsilon(c))}$, where $F(x)=Ax+b$, for some $A\in GL(2,\Z)$ and $b\in\R^2$. So by combining the inverse matrix of $\begin{pmatrix}0&1\\-1&-1\end{pmatrix}$ with an appropriate translation, we obtain a symplectic embedding $B^4(w_1)\hookrightarrow B^4((1+\epsilon)w_0)\setminus \mathbb{X}_\Omega$. Analogously we have symplectic embeddings
\[B^4(w_k)\hookrightarrow B^4((1+\epsilon)w_0)\setminus \left(\mathbb{X}_\Omega\sqcup\bigsqcup_{i=1}^{k-1} B^4(w_i)\right).\]
By induction, we obtain \eqref{eq:bp}.

\begin{theorem}[Cristofaro-Gardiner \cite{CristofaroGardiner2019SymplecticEF}]\label{thm:concconv}
Let $\mathbb{X}_\Omega$ be a weakly convex toric domain and let $(w_0;w_1,w_2,\dots)$ be its weight sequence. Then $B^4(a)\hookrightarrow \mathbb{X}_\Omega$ if, and only if
\[B^4(a)\sqcup\bigsqcup_{i= 1}^{\infty} B^4(w_i)\hookrightarrow B^4(w_0).\]
\end{theorem} 

We can now prove Proposition \ref{prop:ballemb} for $c>1$.
\begin{proof}[Proof of Proposition \ref{prop:ballemb} for $c>1$]
First assume that $1< c \le \beta^{-1}(4\pi)$, which implies that $2\pi< w(c)=\beta(c)\le 4\pi$. From Proposition \ref{ganonincreasing} we know that the toric domain $\mathbb{X}_{\Omega_c}$ is weakly convex. Since $\Omega_c$ is symmetric by reflections about the line $x=y$, it follows that $w_0=2t$, where $(t,t)$ is the intersection of the curve \eqref{resultoblate} with the line $x=y$. So $w_0=2g_c(0)=2\beta(c)$. It follows from Remark \ref{rmk:verthor} that the curve \eqref{resultoblate} has a vertical and a horizontal tangent at the point $(\beta(c),\beta(c))$. So $w_1=w_2=\beta(c)$. The slope of the tangent line to the curve \eqref{resultoblate} at $(2\pi,0)$ is \[\frac{g_c'(1)}{g_c'(1)+2\pi}=\frac{-2\pi c}{-2\pi c+2\pi}=\frac{c}{c-1}>1.\]
So $w_3=\beta(c)-2\pi$ and $w_3\geq 0$ by our assumption on $c$. Let $\epsilon>0$. We now construct a ball packing
\[B^4(\beta(c))\sqcup\bigsqcup_{i=1}^{\infty} B^4(w_i)\hookrightarrow B^4((1+\epsilon)w_0).\]
It is clearly enough to find an embedding of the interiors of the triangles $T(w_i)$ and $T(\beta(c))$ into $T(w_0)$, up to affine equivalence. 
We start by leaving all of the triangles coming from $\Omega_2$ where they are found. We also leave the triangle $T(\beta(c))$ corresponding to $w_1$ in its original place in $\Omega_1$. Now we take $T(w_3)$ to the triangle with vertices $(4\pi-\beta(c),2\pi)$, $(2\pi,2\pi)$ and $(\beta(c),\beta(c))$. Here the affine equivalence is given by \[x\mapsto \begin{pmatrix}1&1\\0&1\end{pmatrix}x+\begin{pmatrix}2\pi-\beta(c)\\ 2\pi\end{pmatrix}.\]
Since $\beta\leq 4\pi$ for $c\leq \beta^{-1}(4\pi)$, we can see that the point $(4\pi-\beta(c),2\pi)$ lies inside $\Omega_c$. The remaining triangles in $\Omega_1$ fit into the triangle with vertices $(2\pi,0)$, $(\beta(c),\beta(c)-2\pi)$ and $(\beta(c),\beta(c))$, which is affinely equivalent to the triangle with vertices $(0,0)$, $(0,2\pi)$ and $(\beta(c)-2\pi,2\pi)$, where the affine equivalence is given by
\[x\mapsto\begin{pmatrix}-1&0\\1&-1\end{pmatrix}x+\begin{pmatrix}\beta(c)\\ 0\end{pmatrix}.\] 
Notice that all these remaining triangles are contined in the area of $\Omega_c$ above the line $y=x$, minus the image of $T(w_3)$ because $\beta(c)-2\pi\leq 2\pi$.\\
So we set the remaining triangles from $\Omega_1$ in the latter triangle. Finally we can place $T(\beta(c))$ in the triangle with vertices $(0,0)$, $(\beta(c),0)$ and $(\beta(c),\beta(c))$, yielding the desired ball packing, see Figure \ref{fig:bp}.
It follows from Theorem \ref{thm:concconv} that \begin{equation}\label{eq:gw}B^4(\beta(c))\hookrightarrow  \mathbb{X}_{\Omega_c}.\end{equation}

\begin{figure}[H]
%\centering
\begin{subfigure}[t]{0.5\textwidth}
%\centering
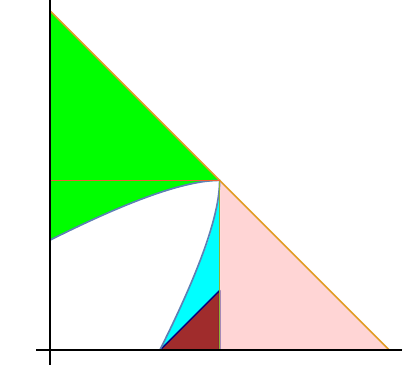
\caption{The weight sequence of $\Omega_c$}
\end{subfigure}
\begin{subfigure}[t]{0.5\textwidth}
%\centering
%% Accompanies image file 'ball packing ex.pdf' (pdf, eps, ps)
\begingroup%
  \makeatletter%
  \providecommand\color[2][]{%
    \errmessage{(Inkscape) Color is used for the text in Inkscape, but the package 'color.sty' is not loaded}%
    \renewcommand\color[2][]{}%
  }%
  \providecommand\transparent[1]{%
    \errmessage{(Inkscape) Transparency is used (non-zero) for the text in Inkscape, but the package 'transparent.sty' is not loaded}%
    \renewcommand\transparent[1]{}%
  }%
  \providecommand\rotatebox[2]{#2}%
  \newcommand*\fsize{\dimexpr\f@size pt\relax}%
  \newcommand*\lineheight[1]{\fontsize{\fsize}{#1\fsize}\selectfont}%
  \ifx\svgwidth\undefined%
    \setlength{\unitlength}{199.4066887bp}%
    \ifx\svgscale\undefined%
      \relax%
    \else%
      \setlength{\unitlength}{\unitlength * \real{\svgscale}}%
    \fi%
  \else%
    \setlength{\unitlength}{\svgwidth}%
  \fi%
  \global\let\svgwidth\undefined%
  \global\let\svgscale\undefined%
  \makeatother%
  \begin{picture}(1,0.91174081)%
    \lineheight{1}%
    \setlength\tabcolsep{0pt}%
    \put(0,0){\includegraphics[width=\unitlength,page=1]{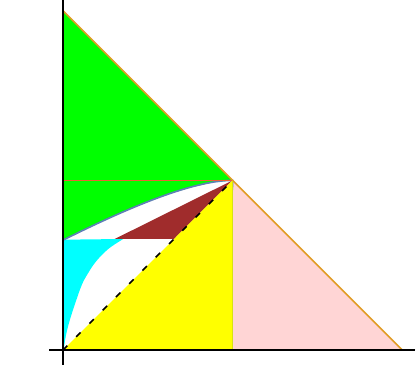}}%
    \put(0.02404594,0.31766592){\color[rgb]{0,0,0}\makebox(0,0)[lt]{\lineheight{1.25}\smash{\begin{tabular}[t]{l}$2\pi$\end{tabular}}}}%
    \put(-0.00416274,0.43990354){\color[rgb]{0,0,0}\makebox(0,0)[lt]{\lineheight{1.25}\smash{\begin{tabular}[t]{l}$\beta(c)$\end{tabular}}}}%
    \put(0.51299645,0.0073705){\color[rgb]{0,0,0}\makebox(0,0)[lt]{\lineheight{1.25}\smash{\begin{tabular}[t]{l}$\beta(c)$\end{tabular}}}}%
    \put(0.3155356,0.13901094){\color[rgb]{0,0,0}\makebox(0,0)[lt]{\lineheight{1.25}\smash{\begin{tabular}[t]{l}$T(\beta(c))$\end{tabular}}}}%
    \put(0.63461884,0.1799146){\color[rgb]{0,0,0}\makebox(0,0)[lt]{\lineheight{1.25}\smash{\begin{tabular}[t]{l}$T(w_1)$\end{tabular}}}}%
    \put(0,0){\includegraphics[width=\unitlength,page=2]{ball_packing_ex.pdf}}%
    \put(0.60176909,0.58220549){\color[rgb]{0,0,0}\makebox(0,0)[lt]{\lineheight{1.25}\smash{\begin{tabular}[t]{l}$T(w_3)$\end{tabular}}}}%
  \end{picture}%
\endgroup%

\caption{The ball packing into $B^4(w_0)$}
\end{subfigure}
\caption{The construction of the ball packing}\label{fig:bp}
\end{figure}

Now suppose that $c>\beta^{-1}(4\pi)$. Then $w(c)=4\pi=\beta(\beta^{-1}(4\pi))$. Moreover, recall from \eqref{eq:int3} that $g_c(j)$ is an increasing funcion of $c$ for every $j\in[0,1]$. So \[a\le b\Rightarrow \mathbb{X}_{a}\subset\mathbb{X}_b.\] It follows that $\mathbb{X}_{\Omega_{\beta^{-1}(4\pi)}}\subset \mathbb{X}_{\Omega_{c}}$. Using \eqref{eq:gw} for $\beta^{-1}(4\pi)$ we conclude that
\[B^4(w(c))=B^4(4\pi)\hookrightarrow  \mathbb{X}_{\Omega_{\beta^{-1}(4\pi)}}\subset \mathbb{X}_{\Omega_{c}}.\]

\end{proof}
\bibliographystyle{unsrt}
\bibliography{biblio}
\end{document}